\documentclass{amsart}
\usepackage{amsmath,amssymb,amsthm,amsfonts,amscd,mathrsfs,mathtools}
\usepackage[dvipsnames]{xcolor}
\usepackage{tikz-cd}
\usepackage[hyphens]{url}
\usepackage{caption}
\usepackage{hyperref}
\hypersetup{
linkbordercolor=blue,
pdfborderstyle={/S/U/W 1}
}
\usepackage{overpic}
\usepackage{color}
\usepackage{pinlabel}
\usepackage{longtable}
\usepackage{multicol}

\theoremstyle{plain}
    \newtheorem{thm}{Theorem}[section]
    \newtheorem{lem}[thm]   {Lemma}
    \newtheorem{cor}[thm]   {Corollary}
    \newtheorem{prop}[thm]  {Proposition}
        \newtheorem{remark}[thm] {Remark}
        \newtheorem{question}[thm]  {Question}

\theoremstyle{definition}
    \newtheorem{defn}[thm]  {Definition}
    \newtheorem{conj}[thm]{Conjecture}
    
    \newtheorem{nota}[thm]{Notation}

\newcommand{\cA}{\mathcal{A}} %arc complex
\newcommand{\I}{\mathcal{I}} %an interval on dM
\newcommand{\cJ}{\mathcal{J}} %Johnson's filtration
\newcommand{\M}{\mathcal{M}} %a manifold
\newcommand{\Q}{\mathbb{Q}} %rational numbers
\newcommand{\fp}{\mathfrak{p}} %augmentation
\newcommand{\R}{\mathbb{R}} %real numbers
\newcommand{\fS}{\mathfrak{S}} %symmetric groups
\newcommand{\ft}{\mathfrak{t}} %a tuple
\newcommand{\bfT}{\mathbf{T}} %a torus
\newcommand{\bT}{\bar{T}} %closed T model for open surfaces
\newcommand{\fw}{\mathfrak{w}} %weight
\newcommand{\U}{\mathcal{U}}
\newcommand{\V}{\mathcal{V}}
\newcommand{\cX}{\mathcal{X}}
\newcommand{\Z}{\mathbb{Z}} %integers

\newcommand{\uP}{\underline{P}}

\newcommand{\uU}{\underline{U}}
\newcommand{\uV}{\underline{V}}

\newcommand{\Ch}{\mathrm{Ch}} %a chain complex
\newcommand{\tCh}{\tilde\Ch} %a reduced chain complex
\newcommand{\del}{\partial}
\newcommand{\sing}{\mathrm{sing}}
\newcommand{\cell}{\mathrm{cell}}
\newcommand{\PL}{\mathrm{PL}} %Poincare-Lefschetz duality quasi-isomorphism
\newcommand{\tPhi}{\tilde\Phi}
\newcommand{\PBr}{\mathrm{PBr}} %pure braid group
\newcommand{\hPBr}{\widehat{\PBr}} %fat pure braid group

\newcommand{\pa}[1]{\left(#1\right)}

\newcommand{\sca}[1]{\left<#1\right>}

\newcommand{\set}[1]{\left\{#1\right\}}

\renewcommand{\phi}{\varphi}
\renewcommand{\epsilon}{\varepsilon}

\DeclareMathOperator{\Mor}{Mor} %Moriyama group
\DeclareMathOperator{\Diff}{Diff} %group of diffeomorphisms
\DeclareMathOperator{\Homeo}{Homeo} %group of homeomorphisms
\DeclareMathOperator{\Id}{Id} %identity
\DeclareMathOperator{\Hom}{Hom} %homomorphisms
\DeclareMathOperator{\Arc}{Arc} %arc complex
\DeclareMathOperator{\Emb}{Emb} %space of embeddings
\DeclareMathOperator{\Ind}{Ind} %induction
\providecommand{\Tor}{\mathrm{Tor}}

\begin{document}

\title[Mapping class group actions on configuration spaces]{Mapping class group actions on configuration spaces and the Johnson filtration}  

\author{Andrea Bianchi}
% \thanks{Andrea Bianchi was supported by the Danish National Research Foundation through the Centre for Geometry and Topology (DNRF151) and the European Research Council under the European Union’s Horizon 2020 research and innovation programme (grant agreement No. 772960)}
\email{anbi@math.ku.dk}
\address{Department of Mathematical Sciences, University of Copenhagen \newline
Universitetsparken 5, Copenhagen, 2100, Denmark}  

 \author{Jeremy Miller}
%  \thanks{Jeremy Miller was supported in part by NSF grant DMS-1709726 and a Simons Foundation Collaboration Grants for Mathematicians}
  \email{jeremykmiller@purdue.edu}
\address{Purdue University, Department of Mathematics\newline
150 North University, West Lafayette IN, 47907, USA}

\author{Jennifer C. H. Wilson}
\email{jchw@umich.edu}
\address{University of Michigan, Department of Mathematics \newline
530 Church St, Ann Arbor MI, 48109, USA}
% \thanks{Jennifer Wilson was supported in part by NSF grant DMS-1906123}

\subjclass[2020]{
55R80 %Discriminantal varieties and configuration spaces in algebraic topology
57K20 %2-dimensional topology (including mapping class groups of surfaces, Teichm¨uller theory, curve complexes, etc.)
}

\date{\today}

\begin{abstract}
Let $F_n(\Sigma_{g,1})$ denote the configuration space of $n$ ordered points on the surface $\Sigma_{g,1}$ and let $\Gamma_{g,1}$ denote the mapping class group of $\Sigma_{g,1}$. We prove that the action of $\Gamma_{g,1}$ on $H_i(F_n(\Sigma_{g,1});\mathbb{Z})$ is trivial when restricted to the $i$\textsuperscript{th} stage of the Johnson  filtration $\mathcal{J}(i)\subset \Gamma_{g,1}$. We give examples showing that $\mathcal{J}(2)$ acts nontrivially on $H_3(F_3(\Sigma_{g,1}))$ for $g\ge2$, and provide two new conceptual reinterpretations of a certain group introduced by Moriyama.
\end{abstract}

\maketitle

\section{Introduction}

\subsection{Statement of results}
In this paper we study the action of mapping class groups on the homology of configuration spaces.  Given a topological space $Z$, let
\[
 F_n(Z)=\set{(z_1,\dots,z_n)\in Z^n\,|\, z_i\neq z_j\ \text{ for } i\neq j}
\]
denote the configuration space $n$ distinct ordered points in $Z$.

Let $\M=\Sigma_{g,1}$ denote a smooth, compact, connected, oriented surface of genus $g$ with one boundary component $\del\M$.
Let $\Homeo(\M,\del\M)$ denote the topological group of homeomorphism of $\M$ that fix $\del \M$ pointwise. Note that a homeomorphism fixing the boundary is automatically orientation-preserving.
The action of $\Homeo(\M,\del\M)$ on $ F_n(\M)$ factors through an action of the mapping class group \[
\Gamma(\M,\del\M)\colon = \pi_0\Homeo(\M,\del\M).
 \] We investigate the following question.

\begin{question}
What is the kernel of the action of $\Gamma(\M,\del\M)$ on $H_i(F_n(\M))$?
\end{question}

Let $\cJ(i)\subset\Gamma(\M,\del\M)$ denote the $i$\textsuperscript{th} stage of the Johnson filtration. That is, $\cJ(i)$ is the kernel of the action of $\Gamma(\M,\del\M)$ on $\pi_1(\M)/\gamma_i\pi_1(\M)$, where $\gamma_i\pi_1(\M)$ is the $i$\textsuperscript{th} term in the lower central series of $\pi_1(\M)$. Our main theorem is the following.

\begin{thm}
 \label{thm:main}
Let $\M$ be a compact orientable surface with one boundary component and let $n,i \geq 0$.
Then the $i$\textsuperscript{th} stage of the Johnson filtration of the mapping class group,
$\cJ(i)\subset\Gamma(\M,\del\M)$, acts trivially on $H_i(F_n(\M))$ and $H^i(F_n(\M))$, the $i$\textsuperscript{th}
homology and cohomology of the ordered configuration space of $n$ points in $\M$.
\end{thm}
We remark that if $\mathring{\M}\subset\M$ denotes the interior of $\M$, then the natural map $F_n(\mathring{\M}) \to F_n(\M)$  is a homotopy equivalence. For technical reasons it will be convenient to work with the space $F_n(\mathring{\M})$, so we shall focus on this space in the rest of the paper.

Theorem \ref{thm:main} is sharp for small values of $i$ in the sense that there are examples of mapping classes in $\cJ(i-1)$ acting nontrivially on $H_i(F_n(\mathring{\M}))$. Basically by definition the kernel of the action of $\Gamma(\M,\del\M)$ on $H_1(F_1(\mathring{\M})) \cong H_1(\M)$ is the Torelli group $\cJ(1)$. In particular, $\cJ(0)$ does not act trivially on $H_1(\M)$ for $g \geq 1$. The first author \cite[Page 32]{Bianchi} showed that $\cJ(1)$ does not act trivially on $H_2(F_2(\mathring{\M}))$ for $g \geq 2$. A similar result for configuration spaces of closed
surfaces was obtained in \cite{Looijenga}.
In Propositions \ref{H3F3} and \ref{H3F3g2}, we show that $\cJ(2)$ does not act trivially on $H_3(F_3(\mathring{\M}))$ for $g \geq 2$.

We note that $\cJ(i)$ does not coincide with the entire kernel of the action of $\Gamma(\M,\del\M)$ on $H_i(F_n(\mathring{\M}))$ in general. For example, when $n>i$, $H_i(F_n(\mathring{\M})) \cong 0$ and so the kernel of the action of $\Gamma(\M,\del\M)$ on $H_i(F_n(\mathring{\M}))$ is all of $\Gamma(\M,\del\M)=\cJ(0)$. Even when $H_i(F_n(\mathring{\M}))$ is nonzero, the kernel of the action of $\Gamma(\M,\del\M)$ on $H_i(F_n(\mathring{\M}))$ is generally larger than $\cJ(i)$. For example, a Dehn twist around a curve parallel to $\del\M$ acts trivially on $H_i(F_n(\mathring{\M}))$ but is not in $\cJ(i)$ for $i>2$; see Subsection \ref{subsec:bDehnJ2}. We conjecture (Conjecture \ref{6.3}) that the subgroup generated by $\cJ(i)$ and this Dehn twist is exactly the kernel of the action of $\Gamma(\M,\del\M)$ on $H_i(F_n(\mathring{\M}))$ for $i \leq n$.

Finally, we mention that \cite{Stavrou} has recently shown, for the unordered configuration spaces $C_n(\M):=F_n(\M)/\fS_n$, that $\cJ(2)$ acts trivially on all rational homology groups $H_i(C_n(\M) ;\Q)$.

\subsection{Relationship to the work of Moriyama}
\label{subsec:relationship}

The inspiration for our paper is a related theorem of Moriyama \cite{Moriyama} concerning the following quotient of configuration space.

\begin{defn} \label{defMprime}
Fix $p_0 \in \del\M$, let $\I=\del\M \setminus \{p_0\} \cong (0,1)$, and let $\M'=\mathring{\M} \cup \I$. For $n\ge0$ let $F_{n,1}(\M',\I)$ denote the subspace of $F_{n}(\M')$
containing all configurations for which at least one point lies in $\I$. 
\end{defn}

Moriyama proves that the reduced cohomology
$\widetilde H^*(F_n(\M')/F_{n,1}(\M',\I))$ is concentrated in degree $*=n$, and that
the kernel of the action of $\Gamma(\M,\del\M)$ on the reduced cohomology group
$\widetilde H^n(F_n(\M')/F_{n,1}(\M',\I))$ is precisely $\cJ(n)$. See also Theorem \ref{thm:Moriyama}
for an equivalent statement.

Our Theorem \ref{thm:main} is a version of Moriyama's result for the classical configuration space
$F_n(\mathring{\M})$, which is homotopy equivalent to $F_n(\M')$, but clearly not to the quotient $F_n(\M')/F_{n,1}(\M',\I)$ for $n>1$.
Unfortunately, unlike Moriyama, we are unable to determine the kernel completely.

Our result depends on Moriyama's work.  In Proposition \ref{prop:main}, we describe a chain complex computing the (co)homology of
$F_n(\mathring{\M})$ built out of (co)chains on configuration space of $\R^2$ and Moriyama's groups
$\widetilde H^n(F_n(\M')/F_{n,1}(\M',\I))$.

In Appendix \ref{Appendix}, we give two different conceptual reinterpretations of Moriyama's groups.  In Propositions \ref{prop:main} and Corollary \ref{cor:hypertor}, we explain how Moryiama's groups are dual to the syzygies of chains of configurations of a surface when viewed as a module over chains of configurations of a disc.  In Theorem \ref{thm:Moriyamaarccomplex}, we prove that Moriyama's groups are linear dual to the pure-surface-braid-group-equivariant homology of the spherical arc complexes of Hatcher--Wahl \cite[Proposition 7.2]{hatcherwahl}. These spaces appear in work of the second and third author on higher order representation stability \cite{MW1} and our proof of Theorem \ref{thm:main} was inspired by higher order representation stability.

\subsection{Proof sketch and paper outline}
We prove Theorem \ref{thm:main} by constructing chain complexes computing the (co)homology of
$F_n(\mathring{\M})$ built out of Moriyama's groups and chains on configuration spaces of $\R^2$. Using this, our results follow almost immediately as the mapping class group of a surface acts trivially on configurations supported in a disk. In Section \ref{Sec2}, we review background on the Johnson filtration, Moriyama's work, and actions of different kinds of maps on configuration spaces. In Section \ref{sec3}, we describe a cell structure on configuration spaces in the spirt of Fox--Neuwirth and Fuchs' stratifications \cite{FoxNeuwirth, Fuchs:CohomBraidModtwo} which gives the desired chain complex model. In Section \ref{sec4}, we describe how to homotope homeomorphisms to respect this cell structure so that elements of the mapping class group act on our chain complex. In Section \ref{sec5}, we prove the main theorem. In Section \ref{sec:sharpness}, we describe some examples relating to the sharpness/non-sharpness of the main theorem. In Appendix \ref{Appendix}, we give some alternative descriptions of the groups appearing in Moriyama's work.

\subsection{Acknowledgments} We would like to thank Benson Farb for suggesting this problem to the second author. We would like to thank Manuel Krannich and  Andrew Putman for helpful conversations. We thank the referee for helpful comments.

Andrea Bianchi was supported by the Danish National Research Foundation through the Centre for Geometry and Topology (DNRF151) and the European Research Council under the European Union Horizon 2020 research and innovation programme (grant agreement No. 772960).

Jeremy Miller was supported in part by NSF grant DMS-1709726 and a Simons Foundation Collaboration Grants for Mathematicians.

Jennifer Wilson was supported in part by NSF grant DMS-1906123.

\section{Preliminaries} \label{Sec2}
\subsection{The Johnson filtration}
We fix a basepoint $p_0\in\del\M$ and
abbreviate $\pi=\pi_1(\M,p_0)$ and $\Gamma=\Gamma(\M,\del\M)$.

Recall that $\pi$ is a free group on $2g$ generators, and that it admits a lower central series
\[
 \pi=\gamma_0\pi\supset\gamma_1\pi\supset\gamma_2\pi\supset\dots,
\]
where $\gamma_0\pi$ is defined to be $\pi$, and for $i\ge0$ we define $\gamma_{i+1}\pi=[\pi,\gamma_i\pi]$.
All groups $\gamma_i\pi$ are characteristic subgroups of $\pi$.
\begin{defn}
 \label{defn:Johnsonfiltration}
 For $i\ge0$ we define $\cJ(i)\subset\Gamma$ to be the kernel of the action of $\Gamma$ by automorphisms
 of the lower central quotient $\pi/\gamma_i\pi$.
\end{defn}
For example, $\cJ(0)=\Gamma$, as the quotient $\pi/\gamma_0\pi$ is the trivial group; and $\cJ(1)$
is the Torelli group, i.e. the subgroup of $\Gamma$ acting trivially on the first homology of $\M$.

\subsection{Action of homeomorphisms on configuration spaces}
\label{subsec:actionGamma}
The topological group $\Homeo(\M,\del\M)$ acts continuously on the following spaces
\begin{itemize}
 \item on $\M$, tautologically;
 \item on $\M^n$, by taking the $n$-fold diagonal action of the previous example;
 \item on $F_n(\mathring{\M})$, by restriction of the previous example.
\end{itemize}
Note that the action is by homeomorphisms in all cases (indeed a group can only act by homeomorphisms on a space!).
Homotopic homeomorphisms of $\M$ give rise to homotopic homeomorphisms of $F_n(\mathring{\M})$. In particular
the action of $\Homeo(\M,\del\M)$ on the homology and cohomology groups of $F_n(\mathring{\M})$ factors through
an action of the mapping class group $\Gamma$ on $H_*(F_n(\mathring{\M}))$ and $H^*(F_n(\mathring{\M}))$ respectively.
In the entire article we refer to this action.

\subsection{Naturality of quasi-isomorphisms}
\label{subsec:naturalquasiiso}
We recollect some basic facts about Poincare-Lefschetz duality, equivalence of cellular and singular
homology, and naturality with respect to maps of topological spaces.

Let $\cX$ be an oriented, possibly non-compact manifold of dimension $N\ge0$ without boundary.
We will denote by $\cX^\infty$
the one-point compactification of $X$, with basepoint the point at infinity $\infty$. If $\cX$ is compact, then $\cX^\infty=\cX\sqcup\set{\infty}$.

Poincare-Lefschetz
duality can be expressed as an isomorphism of singular ho\-mo\-lo\-gy/\-co\-homo\-logy groups
\[
 \PL\colon H^{N-*}_{\sing}(\cX) \overset{\cong}{\longrightarrow} H_*^{\sing}(\cX^\infty,\infty)
\]
or as a quasi-isomorphism (canonical up to chain homotopy)
of singular chain and cochain complexes
\[
 \PL\colon \Ch^{N-*}_{\sing}(\cX) \overset{\simeq}{\longrightarrow} \Ch_*^{\sing}(\cX^\infty,\infty).
\]
Suppose now that $\cX^\infty$ is endowed with a cell decomposition, with $\infty$ being a zero-cell, and denote
by $\Ch_*^{\cell}(\cX^\infty,\infty)$ the cellular chain complex of $\cX^\infty$ relative to $\infty$, and by
$H_*^{\cell}(\cX^\infty,\infty)$ the relative cellular homology. The standard comparison between
cellular and singular theories takes the form of an isomorphism of homology groups
\[
 I^{\sing}_{\cell}\colon H_*^{\sing}(\cX^\infty,\infty) \overset{\cong}{\longrightarrow} H_*^{\cell}(\cX^\infty,\infty)
\]
or of a quasi-isomorphism (canonical up to chain homotopy)
of chain complexes
\[
 I^{\sing}_{\cell}\colon \Ch_*^{\sing}(\cX^\infty,\infty) \overset{\simeq}{\longrightarrow} \Ch_*^{\cell}(\cX^\infty,\infty).
\]

If $\phi\colon \cX\to \cX$ is an orientation-preserving homeomorphism, then there is an induced homeomorphism
$\phi^{\infty}\colon \cX^\infty\to \cX^\infty$. Suppose moreover that $\psi\colon \cX^\infty\to \cX^\infty$ is a cellular
approximation of $\phi^\infty$, so in particular $\psi$ and $\phi^\infty$ are homotopic; then the following
diagram of chain complexes is homotopy commutative
\begin{equation}
 \begin{tikzcd}
  \Ch^{N-*}_{\sing}(\cX)\ar[d,"(\phi^{-1})^*"]\ar[r,"\PL"] & \Ch^{\sing}_*(\cX^\infty,\infty)\ar[d,"\phi^\infty_*"]\ar[r,"I^{\sing}_{\cell}"] & \Ch^{\cell}_*(\cX^\infty,\infty) \ar[d,"\psi_*"]\\
  \Ch^{N-*}_{\sing}(\cX) \ar[r,"\PL"] & \Ch^{\sing}_*(\cX^\infty,\infty)\ar[r,"I^{\sing}_{\cell}"] & \Ch^{\cell}_*(\cX^\infty,\infty)  
 \end{tikzcd} \label{DiagramPLD}
\end{equation}
and induces a commutative diagram of homology/cohomology groups.

The upshot of this discussion is that
we can study the behaviour of $\phi\colon \cX\to \cX$ in cohomology by understanding the behaviour
of a cellular approximation $\psi$ of $\phi^\infty$ on the cellular homology of $(\cX^\infty,\infty)$.

We will apply these remarks to $\cX=F_n(\mathring{\M})$, which is a $2n$-manifold \emph{without boundary}: this is essential advantage of taking configurations in $\mathring{\M}$ rather than in $\M$. In particular we can study the action of $\Gamma$ on
the homology groups $H_{2n-i}(F_n(\mathring{\M})^\infty,\infty)$ instead of the action
on the cohomology groups $H^i(F_n(\mathring{\M}))$.

\subsection{Moriyama's results}
Let $p_0\in\del\M$ be a point. Moriyama considers two subspaces $\Delta_n(\M)$ and $A_n(\M)$ of $\M^n$, namely
\[
 \Delta_n(\M)=\set{(z_1,\dots,z_n)\in \M^n\,|\, z_i=z_j\mbox{ for some } i\neq j},
\]
\[
 A_n(\M)=\set{(z_1,\dots,z_n)\in \M^n\,|\, z_i=p_0\mbox{ for some } i}.
\]
He proves then the following theorem.

\begin{thm}[\cite{Moriyama}, Theorem A]
 \label{thm:Moriyama}
 The relative homology group
 \[
 H_i(\M^n,\Delta_n(\M)\cup A_n(\M))
 \]
 vanishes for $i\neq n$, and is free abelian of rank
 $(2g)\cdot(2g+1)\dots(2g+n-1)$ for $i=n$.
 The action of $\Gamma$ on $H_n(\M^n,\Delta_n(\M)\cup A_n(\M))$ has kernel precisely the subgroup $\cJ(n)\subset\Gamma$.
\end{thm}
Here the action of $\Gamma$ on $H_n(\M^n,\Delta_n(\M)\cup A_n(\M))$ is defined similarly as in Subsection \ref{subsec:actionGamma},
by noting that the action of $\Homeo(\M,\del\M)$ on $\M^n$ preserves both subspaces $\Delta_n(\M)$ and $A_n(\M)$ of $\M^n$. The equivalence of Theorem \ref{thm:Moriyama} with the statement
of Subsection \ref{subsec:relationship} is an application of Poincare-Lefschetz duality
and was already observed by Moriyama. For completeness, we added the argument in Lemma \ref{lem:naturalisoMoriyama} in the Appendix \ref{Appendix}.

\begin{nota}
 For all $n\ge0$ we denote by $\Mor_n$ the $\Gamma$-representation given by $H_n (\M^n,\Delta_n(\M)\cup A_n(\M))$.
\end{nota}

\begin{defn}
 \label{defn:Apn}
We denote by $A'_n(\M)\subset\M^n$ the subspace
\[
  A'_n(\M)=\set{(z_1,\dots,z_n)\in \M^n\,|\, z_i\in\del\M\mbox{ for some } i}.
\]
\end{defn}
Clearly $A_n\subset A'_n$. Moreover the one-point compactification $F_n(\mathring{\M})^\infty$ can be regarded
as the quotient of the space $\M^n$ by the subspace $\Delta^n(\M)\cup A'_n(\M)$. In other words, using also the remarks
from Subsection \ref{subsec:naturalquasiiso}, we can identify $H^*(F_n(\mathring{\M}))$ with
\[
H_*(\M^n,\Delta_n(\M)\cup A'_n(\M))\cong H_*(F_n(\mathring{\M})^\infty,\infty)
\]
as $\Gamma$-representations.

\subsection{Functoriality}
\label{subsec:functoriality}
In the previous subsection we have associated several spaces with our surface $\M$, namely
$\M^n$, $\Delta_n(\M)$, $A_n(\M)$, $A'_n(\M)$, $F_n(\mathring{\M})$ and $F_n(\mathring{\M})^\infty$. Let $f\colon \M\to\M$ be a continuous map;
then $f$ always induces a map $f^n\colon\M^n\to \M^n$; the map $f^n$ restricts to a map:
\begin{itemize}
 \item $\Delta_n(f)\colon\Delta_n(\M)\to\Delta_n(\M)$, always;
 \item $A_n(f)\colon A_n(\M)\to A_n(\M)$, if $f(p_0)=p_0$;
 \item $A'_n(f)\colon A'_n(\M)\to A'_n(\M)$, if $f(\del\M)\subseteq\del\M$;
 \item $F_n(f)\colon F_n(\mathring{\M})\to F_n(\mathring{\M})$, if $f$ is injective (and hence $f(\mathring{\M})\subseteq\mathring{\M}$,
 since $\M$ is a manifold with boundary);
\end{itemize}
and induces a map $F_n(f)^\infty\colon F_n(\mathring{\M})^\infty\to F_n(\mathring{\M})^\infty$,
if $f(\del\M)\subseteq\del\M$.

The previous discussion can be enhanced to an enriched setting; for us it suffices to note that a homotopy $f\colon \M\times[0,1]\to\M$ between the maps
$f_0,f_1\colon\M\to\M$ induces a homotopy $f^n\colon \M^n\times[0,1]\to\M$, restricting to homotopies
\begin{itemize}
 \item $\Delta_n(f)\colon\Delta_n(\M)\times[0,1]\to\Delta_n(\M)$, always;
 \item $A_n(f)\colon A_n(\M)\times[0,1]\to A_n(\M)$, if $f_t(p_0)=p_0$ for all $t\in[0,1]$;
 \item $A'_n(f)\colon A'_n(\M)\times[0,1]\to A'_n(\M)$, if $f_t(\del\M)\subseteq\del\M$ for all $t\in[0,1]$;
 \item $F_n(f)\colon F_n(\mathring{\M})\times[0,1]\to F_n(\mathring{\M})$, if $f_t$ is injective for all $t\in[0,1]$;
\end{itemize}
and inducing a homotopy $F_n(f)^\infty\colon F_n(\mathring{\M})^\infty\times[0,1]\to F_n(\mathring{\M})^\infty$,
if $f_t(\del\M)\subseteq\del\M$ for all $t\in[0,1]$.

One of the advantages of working with the couples $(\M^n, \Delta_n(\M)\cup A_n(\M))$ and $(\M^n,\Delta_n(\M)\cup A'_n)$,
or with the one-point compactification $F_n(\mathring{\M})^\infty$, rather than with the plain configuration space $F_n(\mathring{\M})$,
is that in the first cases one can leverage on functoriality of a much larger class of self-maps of $\M$.

From now on we will abbreviate $\Delta_n=\Delta_n(\M)$, $A_n=A_n(\M)$ and $A'_n=A'_n(\M)$.
We also expand our notation as follows. For a finite set $S$ we denote:
\begin{itemize}
 \item $\M^S$ for the space of all maps of sets $S\to\M$;
 \item $A'_S\subset\M^S$ for the subspace of those maps whose image intersects $\del\M$;
 \item $A_S\subset A'_S$ for the subspace of those maps which have $p_0$ in their image;
 \item $\Delta_S\subset\M^S$ for the subspace of those maps which are not injective;
 \item $F_S(\M)\subset\M^S$ for the subspace of those maps which are injective and have image in $\mathring{\M}$;
 \item $\Mor_S$ for the $\Gamma$-representation $H_{|S|}(\M^S,\Delta_S\cup A_S)$.
\end{itemize}

\section{Cell stratifications of configuration spaces} \label{sec3}
In this section we define a cell stratification of the space $F_n(\mathring{\M})$, or equivalently a
relative cell structure of the couple $(F_n(\mathring{\M})^\infty,\infty)$. This construction is similar in
spirit to the Fox--Neuwirth and Fuchs stratifications of unordered
configuration spaces of $\R^d$ \cite{FoxNeuwirth, Fuchs:CohomBraidModtwo}. A similar construction
was used in \cite{Bianchi}.

\subsection{A convenient model for \texorpdfstring{$\M$}{M}}
We introduce a convenient model $\bT(\M)$ for our surface $\M\cong\Sigma_{g,1}$, arising as a quotient
of the rectangle $[0,2]\times[0,1]$.

\begin{nota}
\label{nota:IJi}
We consider the following dissection of the unit interval $[0,1]$ into $4g$ intervals
denoted $I_i,J_i,I'_i,J'_i$, for $1\le i\le g$: we set
\[
\begin{array}{cc}
I_i=\left[\frac{4i-4}{4g},\frac {4i-3}{4g}\right]; & J_i=\left[\frac{4i-3}{4g},\frac {4i-2}{4g}\right];\\[10pt]
I'_i=\left[\frac{4i-2}{4g},\frac {4i-1}{4g}\right];&  J'_i=\left[\frac{4i-1}{4g},\frac {4i}{4g}\right]. 
\end{array}
\]
We consider, for $1\le i\le g$, the following linear homeomorphism between the above
intervals and the entire interval $[0,1]$:
\[
 \begin{array}{cc}
  \rho^I_i\colon[0,1]\to I_i; & \rho^I_i\colon t\mapsto \frac{4i-4}{4g} + \frac{t}{4g};\\[5pt]
  \rho^J_i\colon[0,1]\to J_i; & \rho^J_i\colon t\mapsto \frac{4i-3}{4g} + \frac{t}{4g};\\[5pt]
  \rho^{I'}_i\colon[0,1]\to I'_i; & \rho^{I'}_i\colon t\mapsto \frac{4i-1}{4g} - \frac{t}{4g};\\[5pt]
  \rho^{J'}_i\colon[0,1]\to J'_i; & \rho^{J'}_i\colon t\mapsto \frac{4i}{4g} - \frac{t}{4g}.
 \end{array}
\]
\end{nota}

\begin{defn}
\label{def:T(M)}
We define $\bT(\M)$ as the space obtained from the rectangle $[0,2]\times[0,1]$ applying the following identifications.
\begin{itemize}
 \item For $1\leq i\leq g$ we identify the right vertical segments
 $2\times I_i$ and $2\times I'_i$ along the homeomorphism $(2,\rho^I_i(t))\mapsto (2,\rho^{I'}_i(t))$, for $t\in [0,1]$.
 \item For $1\leq i\leq g$ we identify the right vertical segments
 $2\times J_i$ and $2\times J'_i$ along the homeomorphism $(2,\rho^J_i(t))\mapsto (2,\rho^{J'}_i(t))$, for $t\in [0,1]$.
\end{itemize}

The space $T(\M)$ is obtained from $\bT(\M)$ by removing the image in the quotient of
the subspace $\set{0}\times [0,1]\cup [0,2]\times\set{0,1}$.
\end{defn}

The space $\bT(\M)$ is homeomorphic to the closed surface $\M\cong\Sigma_{g,1}$, and $T(\M)$ is homeomorphic
to its interior.
The space $\bT(\M)$ is endowed with the following cell decomposition:
\begin{itemize}
 \item there is 1 zero-cell, that we call $p_0$, following Moriyama; it is the image in the quotient
 of the points $2\times \frac{i}{4g}$ for $0\le i\le 4g$;
 \item there are $2g+1$ one-cells, namely:
 \begin{itemize}
  \item there is 1 one-cell $\I$, which is the image in the quotient of the union $\set{0}\times[0,1]\cup[0,2)\times\set{0,1}$;
 \item there are $g$ one-cells, denoted $\U_i$ for $1\leq i\leq g$,
 which are the images in the quotient of the pairs of open segments $\set{1}\times \mathring I_i$ and $\set{1}\times \mathring I'_i$;
 \item there are $g$ one-cells, denoted $\V_i$ for $1\leq i\leq g$,
 which are the images in the quotient of the pairs of open segments
 $\set{1}\times \mathring J_i$ and $\set{1}\times \mathring J'_i$.
 \end{itemize}
 \item there is 1 two-cell $E$, which is the image in the quotient of $(0,2)\times (0,1)$.
\end{itemize}
See Figure \ref{fig:Mopen}, where the space $\bT(\M)$ is represented in the case $g=2$.
Clearly $T(\M)$ is the union of the open cells $E$, $\U_i$ and $\V_i$ for $1\le i\le g$.

\begin{figure}[ht]
 \centering
 \begin{overpic}[width=12cm]{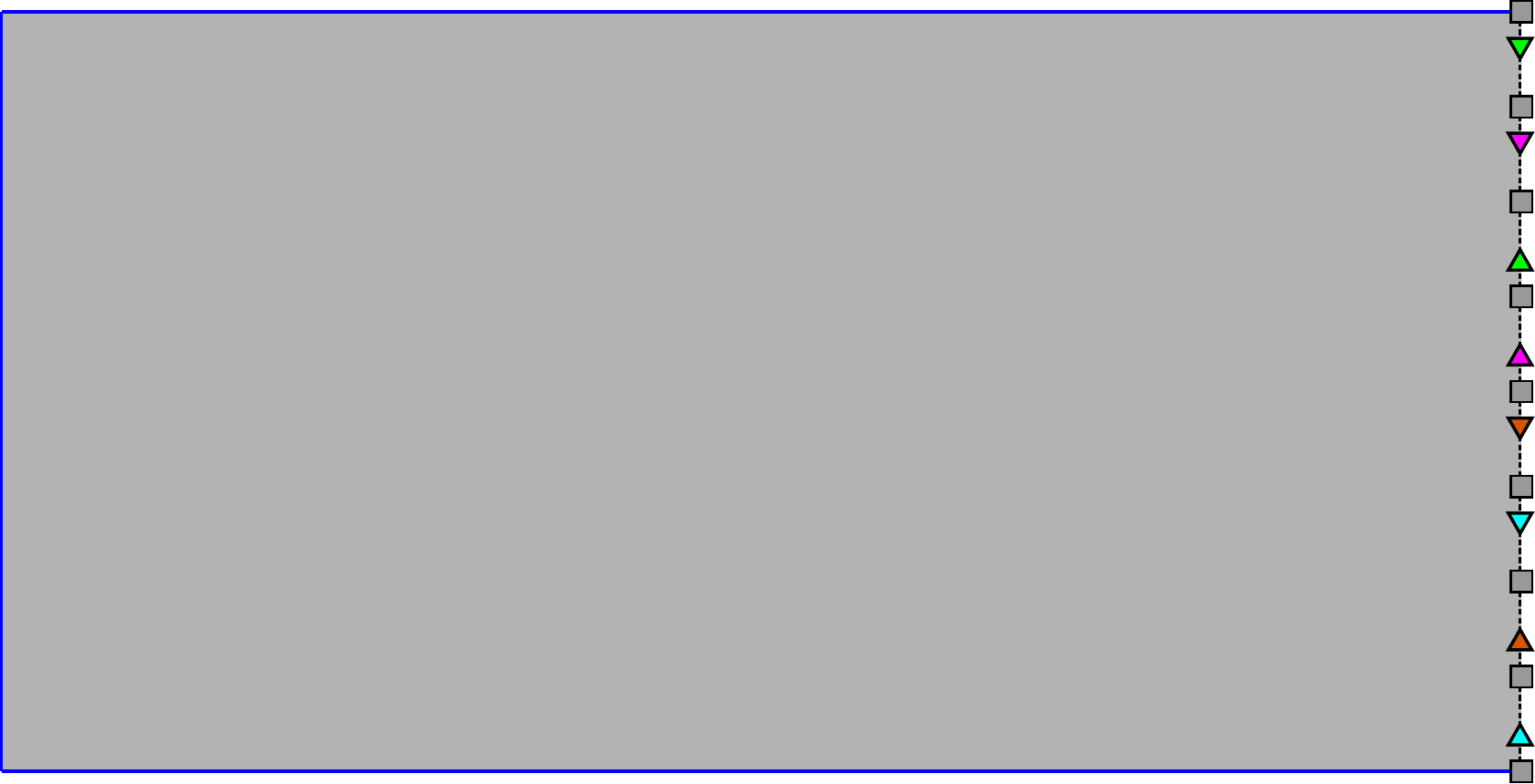}
 \put(50,25){$E$} \put(94,3){$\U_1$}  \put(94,9){$\V_1$}  \put(94,15){$\U_1$} \put(94,21){$\V_1$} \put(94,28){$\U_2$} \put(94,34){$\V_2$} \put(94,40){$\U_2$} \put(94,46){$\V_2$}
 \put(100.5,0){$p_0$}\put(100.5,6){$p_0$}\put(100.5,12){$p_0$}\put(100.5,19){$p_0$}\put(100.5,25){$p_0$}\put(100.5,31){$p_0$}\put(100.5,37){$p_0$}\put(100.5,44){$p_0$}\put(100.5,50){$p_0$}
 \put(30,1.5){$\textcolor{blue}{\I}$}
 \end{overpic}
 \caption{A cell decomposition of $\bT(\M)$}
 \label{fig:Mopen}
\end{figure}
From now on we will not distinguish between the surface $\M$ and its model $\bT(\M)$.

\subsection{Cell stratification of configuration spaces}
We stratify $F_n(\mathring{\M})$ by open cells $e_{\ft}$.

\begin{defn}
\label{defn:et}
Let $S$ be a finite set.
A \emph{tuple} on the set $S$, denoted generically $\ft$,
is a choice of the following set of data:
\begin{itemize}
 \item an integer $\ell\ge0$ called the \emph{length} of the tuple;
 \item a partition of the set $S$ into subsets $P_1,\dots,P_\ell,U_1,\dots,U_g,V_1,\dots,V_g$,
 with $P_1,\dots,P_\ell$ non-empty;
 \item a total order $\prec_{P_i}$, $\prec_{U_i}$ or $\prec_{V_i}$ on each of the previous subsets.
\end{itemize}
We generically write $\ft=(\ell,\uP,\uU,\uV)$, where each underlined letter denotes the sequence of
finite sets, and we omit the sequence of total orders from the notation. We replace $\uP$, $\uU$ or $\uV$
by ``$()$'' whenever we want to emphasize that $\uP$, $\uU$ or $\uV$ consists of no set.
We define the \emph{weight} of $\ft$ to be the cardinality of $S$, and write $\fw(\ft)=|S|$.
The \emph{degree} of $\ft$ is defined as $d(\ft)=\fw(\ft)+\ell$.

For a tuple on $S$, let $e_\ft$ be the subspace of $F_S(\M)$ of configurations $(z_i)_{i\in S}$ of
$|S|$ distinct points in $\mathring{\M}$ such that
\begin{itemize}
 \item for all $1\le i\le g$ the points $z_j$ with $j\in U_i$ lie on the one-cell $\U_i$;
 for all $j,j'\in U_i$, regard $z_j$ and $z_{j'}$ as real numbers in the interior of the interval $I_i$,
 under the identification $I_i\cong \set{2}\times I_i$ followed by the quotient map $[0,2]\times[0,1]\to \M$:
 then $j\prec_{U_i} j'$ if and only if $z_j<z_{j'}$;
 \item for all $1\le i\le g$ the points $z_j$ with $j\in V_i$ lie on the one-cell $\V_i$;
 for all $j,j'\in V_i$, regard $z_j$ and $z_{j'}$ as real numbers in the interior of the interval $J_i$,
 under the identification $J_i\cong \set{2}\times J_i$ followed by the quotient map $[0,2]\times[0,1]\to \M$:
 then $j\prec_{V_i} j'$ if and only if $z_j<z_{j'}$;
 \item there are precisely $\ell$ numbers $0<x_1<\dots<x_\ell<2$ such that each point $z_j\in E$ lies on some
 open segment $\set{x_i}\times(0,1)$, and viceversa each segment $\set{x_i}\times(0,1)$ contains some point $z_j$;
 for all $j\in P_i$ the point $z_j$ lies on $\set{x_i}\times(0,1)$; for all $j,j'\in P_i$, regard $z_j$ and $z_{j'}$
 as real numbers in $(0,1)$ under the identification $\set{x_i}\times(0,1)\cong(0,1)$:
 then $j\prec_{P_i}j'$ if and only if $z_j<z_{j'}$.
\end{itemize}
\end{defn}
See Figure \ref{fig:cell}.

\begin{figure}[ht]
 \centering
 \begin{overpic}[width=12cm]{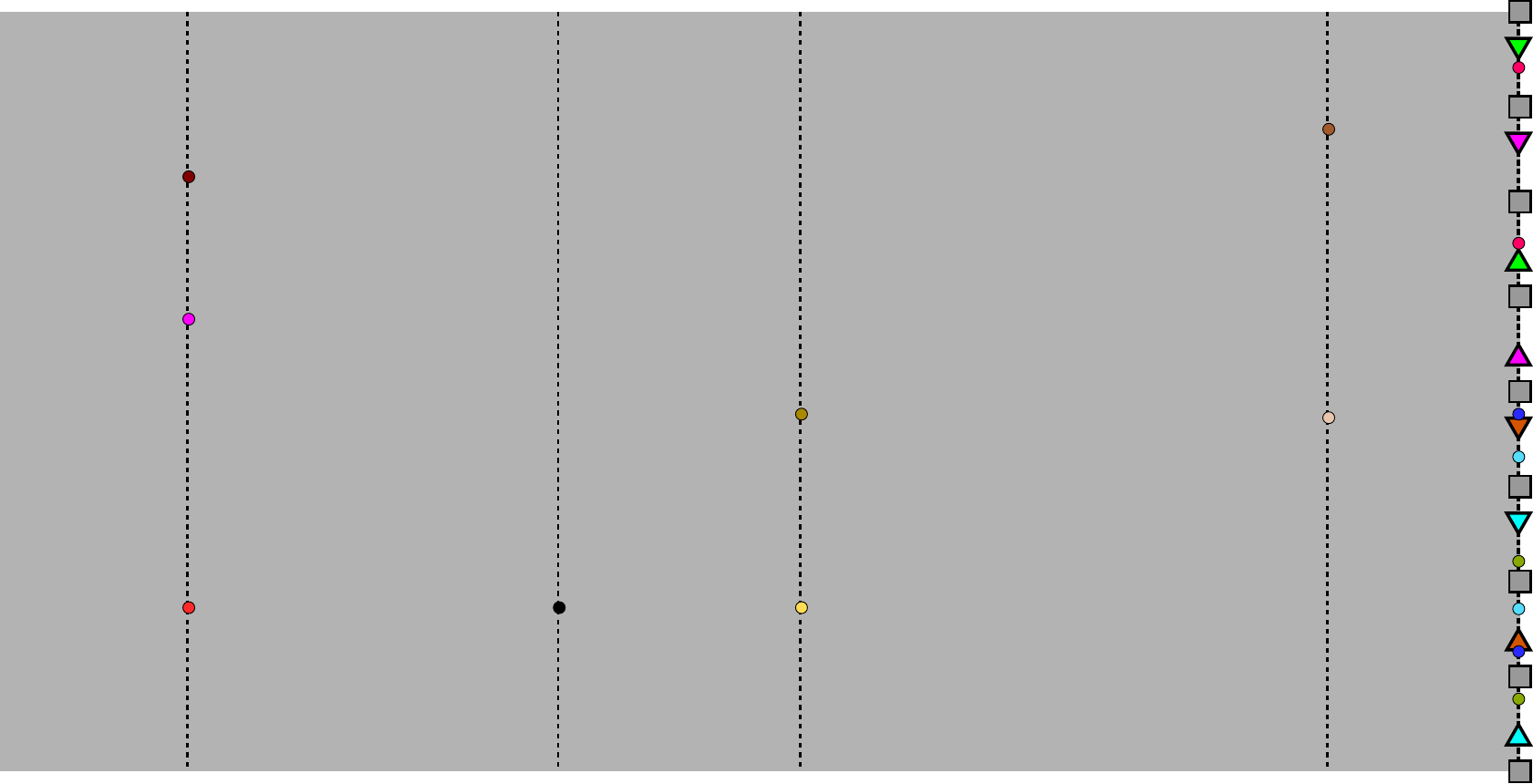}
 \put(58,25){$E$} \put(94,3){$\U_1$} \put(94,9){$\V_1$} \put(94,15){$\U_1$} \put(94,21){$\V_1$} \put(94,27){$\U_2$} \put(94,33){$\V_2$} \put(94,39){$\U_2$} \put(94,45){$\V_2$}
 \put(13,10){$z_{11}$}\put(13,29){$z_9$}\put(13,38){$z_2$}
 \put(37,10){$z_{10}$}
 \put(48,10){$z_{12}$}\put(48,24){$z_3$}
 \put(83,24){$z_5$}\put(83,41){$z_8$}
 \put(100,4){$z_6$}\put(100,8){$z_1$}\put(100,10.5){$z_4$}\put(100,14){$z_6$}
 \put(100,21){$z_4$}\put(100,24){$z_1$}\put(100,35){$z_7$}\put(100,46){$z_7$}
 \end{overpic}
 \caption{A configuration lying in the open cell $e_\ft\subset F_{12}(\M)$, where $\ft=(\ell,\uP,\uU,\uV)$
 satisfies, for instance, $\ell=4$, $11\prec_{P_1} 9$, $P_2=\set{10}$ and 
 $1\prec_{V_1}4$.}
 \label{fig:cell}
\end{figure}

The space $e_\ft$ is homeomorphic to an open ball of dimension $d(\ft)$. More precisely, let $\Delta^\ft$
be the following multisimplex
\[
 \Delta^\ft=\Delta^\ell\times\prod_{i=1}^\ell\Delta^{|P_i|}\times\prod_{i=1}^g(\Delta^{|U_i|}\times\Delta^{|V_i|}).
\]
Then there is a map $\Phi^\ft\colon\Delta^\ft\to\M^S$ with the following properties:
\begin{itemize}
 \item $\Phi^\ft$ maps $\mathring{\Delta}^\ft$ homeomorphically onto $e_\ft$;
 \item $\Phi^\ft$ maps $\del\Delta^\ft$ inside the union of $\Delta_S$, $A'_S$ and all subspaces $e_{\ft'}$
 corresponding to tuples $\ft'$ on $S$ of length $<\ell$.
\end{itemize}
Note that if $\ft'$ and $\ft$ are tuples on $S$ and the length of $\ft'$ is strictly smaller than the length
of $\ft$, then $d(\ft')<d(\ft)$, i.e. the dimension
of the cell $e_{\ft'}$ is lower than the dimension of the cell $e_\ft$. It follows that the cells
$e_\ft$, for $\ft$ varying among tuples on $S$, give a relative cell decomposition for the couple of spaces
$(\M^S,\Delta_S\cup A'_S)$; and taking the quotient by $\Delta_S\cup A'_S$, we obtain a cell decomposition
on $F_S(\M)^\infty$ with a single zero-cell $\infty$.

The map $\Phi^\ft$ is constructed explicitly as follows. First we construct a map
$\tPhi^\ft\colon\Delta^\ft\to([0,2]\times[0,1])^S$, i.e. $|S|$ maps $\tPhi^\ft_j\colon \Delta^\ft\to[0,2]\times[0,1]$, one for each $j\in S$:
\begin{itemize}
 \item for all $j\in P_i$, denote by $1\le\iota\le|P_i|$ the position of $j$ in the total order $\prec_{P_i}$;
 we define $\tPhi^\ft_j=(2x_i)\times y^P_{i,\iota}$, where $x_i\colon \Delta^\ft\to[0,1]$ is the $i$\textsuperscript{th} coordinate
 of the factor $\Delta^\ell$, and $y^P_{i,\iota}$ is the $\iota$\textsuperscript{th} coordinate of the factor $\Delta^{|P_i|}$;
 \item for all $j\in U_i$, denote by $1\le\iota\le|U_i|$ the position of $j$ in the total order $\prec_{U_i}$;
 we define $\tPhi^\ft_j=2\times (\rho^I_i\circ y^U_{i,\iota})$,
 where $y^U_{i,\iota}$ is the $\iota$\textsuperscript{th} coordinate of the factor $\Delta^{|U_i|}$;
 \item for all $j\in V_i$, denote by $1\le\iota\le|V_i|$ the position of $j$ in the total order $\prec_{V_i}$;
 we define $\tPhi^\ft_j=2\times (\rho^J_i\circ y^V_{i,\iota})$,
 where $y^V_{i,\iota}$ is the $\iota$\textsuperscript{th} coordinate of the factor $\Delta^{|V_i|}$.
\end{itemize}
We then let $\Phi^\ft$ be the composition of $\tPhi^\ft$ with the quotient map
$([0,2]\times[0,1])^S\to\M^S$.

\section{Cellular approximation of homeomorphisms}
\label{sec4}

Let $D$ denote a closed disc in $\M$ embedded near the boundary, i.e. $\M$ can be regarded as the boundary
connected sum of $D$ and another subsurface of type $\Sigma_{g,1}$.

In the model $\bT(\M)$ of the surface $\M$, we let $D$ be the left unit square, i.e. the image in the quotient
of $[0,1]\times[0,1]\subset[0,2]\times[0,1]$.

\begin{defn}
 We denote by $\Homeo(\M,\del\M\cup D)\subset\Homeo(\M,\del\M)$ the subgroup of homeomorphisms
 that fix $\del\M\cup D$ pointwise.
\end{defn}
The inclusion $\Homeo(\M,\del\M\cup D)\to \Homeo(\M,\del\M)$ is a homotopy
equivalence, in particular it induces a bijection
\[
\pi_0(\Homeo(\M,\del\M\cup D))\to\pi_0(\Homeo(\M,\del\M)).
\]
We can thus identify $\Gamma$ with the group $\pi_0(\Homeo(\M,\del\M\cup D))$.

\subsection{The map \texorpdfstring{$\tau$}{tau}}
In Section \ref{sec3}, we constructed a cell structure giving a chain complex computing the homology of configuration spaces. In order to prove our results about mapping class groups, we need to show that this cell structure is compatible with the mapping class group action. To do this, we will define a map $\tau\colon \M\to \M$ which is homotopic to the identity,  and use $\tau$ to show the action of an arbitrary homeomorphism in $\Homeo(\M,\del\M\cup D)$  is homotopic to a cellular map.

\label{subsec:themaptau}
\begin{defn}
 \label{defn:tau}
We define a continuous (but not injective) map $\tau\colon \M\to \M$;
it is induced on the quotient by the map
$\tilde\tau\colon[0,2]\times[0,1]\to [0,2]\times[0,1]$ given by the following formula
\[
 \tau(x,y)=\left\{
\begin{array}{ll}
 (2x,y)& \mbox{if } 0\le x\le 1;\\
 (2,y) & \mbox{if } 1\le x\le 2.
\end{array}
 \right.
\] 
\end{defn}
The map $\tau$ expands the disc $D$ horizontally and collapses the dark grey region onto the ``right
side'' of $M$. Note that $\tau$ preserves $\del\M$ but does not fix it pointwise;
note also that $\tau$ is homotopic to the identity of $\M$ through maps that preserve $\del\M$.
See Figure \ref{fig:tau}.

\begin{figure}[ht]
 \centering
 \begin{overpic}[width=12cm]{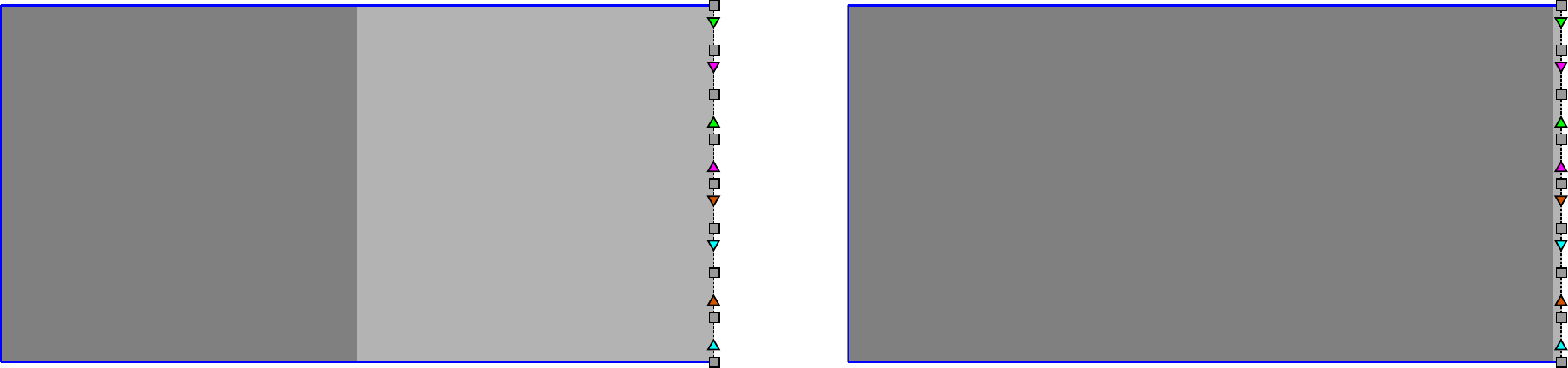}
 \put(13,10){$D$}
 \put(49,10){$\overset{\tau}{\to}$}
 \end{overpic}
 \caption{The map $\tau$ expands $D$ and contracts horizontally the rest of $\M$.}
 \label{fig:tau}
\end{figure}

Let $S$ be a finite set. As remarked in Subsection \ref{subsec:functoriality},
$\tau$ does \emph{not} induce a map on configuration spaces $F_S(\M)\to F_S(\M)$,
since it is not injective. Nevertheless, $\tau$ induces a map $\tau^S\colon\M^S\to\M^S$ preserving all subspaces
$\Delta_S$, $A_S$ and $A'_S$: in the case of $A_S$ we use in particular
that $\tau\colon\M\to\M$ fixes the point $p_0$.
It follows that $\tau$ induces also a map $F_S(\tau)^\infty\colon F_S(\M)^\infty\to F_S(\M)^\infty$.

\begin{lem}
 The map $F_S(\tau)^\infty\colon F_S(\M)^\infty\to F_S(\M)^\infty$ is cellular. More generally,
 let $f\in\Homeo(\M,\del\M\cup D)$. Then
 $F_S(\tau)^\infty\circ F_S(f)^\infty=F_S(\tau\circ f)^\infty$ is a cellular self-map of $F_S(\M)^\infty$.
\end{lem}
\begin{proof}
 It suffices to prove the second statement, since the first follows by considering $f=\Id_{\M}$.
 It is clear that both $F_S(f)^\infty$ and $F_S(\tau)^\infty$ preserve the point $\infty\in F_S(\M)^\infty$, hence
 $F_S(\tau\circ f)^\infty$ preserves the zero-skeleton of $F_S(\M)^\infty$.
 
 Let $\ft=(\ell,\uP,\uU,\uV)$ be a tuple on $S$, and let $(z_i)_{i\in S}\in e_\ft\subset F_S(\M)\subset F_S(\M)^\infty$.
 There are three possibilities:
 \begin{enumerate}
  \item the map $\tau\circ f\colon \M\to \M$ is not injective on the set $\set{z_i}_{i\in S}$, or
  it sends some $z_i\mapsto p_0$; then $F_S(\tau\circ f)^\infty$ sends $(z_i)_{i\in S}$ to $\infty$;
  \item the map $\tau\circ f$ is injective on $\set{z_i}_{i\in S}\cup\set{p_0}$, and some $z_i$ lies in the region
  $[1,2)\times(0,1)\subset \mathring{\M}$; then $F_S(\tau\circ f)^\infty$ sends $(z_i)_{i\in S}$ to a configuration
  in $F_S(\M)\subset F_S(\M)^\infty$ belonging to a cell $e_{\ft'}$ with $\ft'$ being a tuple on $S$ of strictly
  lower length than $S$, in particular $e_{\ft'}$ has strictly lower dimension than $e_{\ft}$;
  \item the map $\tau\circ f$ is injective on $\set{z_i}_{i\in S}\cup\set{p_0}$, and all points $z_i$ lie either in the
  interior of $D$, or on the open one-cells $\U_i$ and $\V_i$; then $F_S(\tau\circ f)^\infty$ sends $(z_i)_{i\in S}$ to a
  configuration  in $F_S(\M)$ belonging to a cell $e_{\ft'}$ with $\ft'$ of the form $(\ell,\uP,\uU',\uV')$
  for some $\uU'$ and $\uV'$.
 \end{enumerate} \end{proof}

A particular instance of the fact that $F_S(\tau\circ f)^\infty$ is cellular is the following.
Let $\ft=(\ell,\uP,(),())$ be a tuple on $S$, where we use the notation from Definition \ref{defn:et};
denote by $e_\ft^{\mathring{D}}\subset e_\ft$ the subspace of configurations
$(z_i)_{i\in S}\in e_\ft$ such that all points $z_i$ lie in $\mathring{D}$. Then $F_S(\tau\circ f)^\infty$
maps $e_\ft^{\mathring{D}}$ homeomorphically onto $e_\ft$, and sends every configuration
in $e_\ft\setminus e_\ft^{\mathring{D}}$ to $\infty$ or to cells $e_{\ft'}$ of strictly lower dimension.

\subsection{Superposition of configurations}
\label{subsec:superposition}
One of the advantages of working with the one-point compactified configuration space $F_S(\M)^\infty$
is that we can allow collisions between points in a configuration: if two points $z_i$ and $z_j$
collide, we simply declare the result to be the ``configuration'' $\infty\in F_S(\M)^\infty$.

We have already encountered this idea in the discussion of functoriality in Subsection \ref{subsec:functoriality},
and exploited it in the previous subsection,
where postcomposition with $F_S(\tau)^\infty$ makes
every map $F_S(f)^\infty$, coming from a homeomorphism $f$, into a cellular map.
In this subsection we use this principle again to construct superposition maps.

\begin{defn}
\label{defn:mu}
 For all bipartitions $Q\sqcup R=S$, possibly with $Q$ or $R$ being empty, we define a map
 \[
  \mu_{Q,R}\colon F_Q(\M)^\infty\wedge F_R(\M)^\infty\to F_S(\M)^\infty.
 \]
 We start with the canonical identification $\tilde \mu_{Q,R}\colon \M^Q\times \M^R\cong \M^S$,
 and note that it restricts to an inclusion
 \[
  \tilde \mu_{Q,R}\colon (\Delta_Q\cup A'_Q)\times \M^R\ \cup\ \M^Q\times (\Delta_R\cup A'_R)\to \Delta_S\cup A'_S.
 \]
 We denote $\mu_{Q,R}$ the induced map on the quotients.
\end{defn}
For any continuous map $\mathfrak{f}\colon \M\to\M$ preserving $\del\M$ the following diagram is commutative on the nose:
\[
 \begin{tikzcd}
  F_Q(\M)^\infty\wedge F_R(\M)^\infty\ar[d,"\mathfrak{f}_*\wedge \mathfrak{f}_*"]\ar[r,"\mu_{Q,R}"] & F_S(\M)^\infty\ar[d,"\mathfrak{f}_*"]\\
  F_Q(\M)^\infty\wedge F_R(\M)^\infty\ar[r,"\mu_{Q,R}"] & F_S(\M)^\infty.\\
 \end{tikzcd}
\]
This holds in particular when $\mathfrak{f}$ has the form $\tau\circ f$, for some homeomorphism $f\in\Homeo(\M,\del\M\cup D)$.
\begin{lem}
 \label{lem:mucellular}
 The map $\mu_{Q,R}\colon F_Q(\M)^\infty\wedge F_R(\M)^\infty\to F_S(\M)^\infty$ is cellular, where we endow
 $F_Q(\M)^\infty\wedge F_R(\M)^\infty$ with the product cell structure.
\end{lem}
\begin{proof}
Let $(z_i)_{i\in Q}\in F_Q(\M)\subset F_Q(\M)^\infty$ and $(z'_j)_{j\in R}\in F_R(\M)^\infty$
be two configurations, and note that $(z_i)_{i\in Q}$ lies in the interior of some cell $e_{\ft}\subset F_Q(\M)$
of dimension $|Q|+\ell_1$, where $\ell_1$ is the cardinality of the set
\[
\set{x\in(0,2)\,|\,\exists i\in Q\, ,\, z_i\in \set{x}\times(0,1)}.
\]
Similarly $(z'_j)_{j\in R}$ lies in the interior of a cell of $F_R(\M)$ of dimension $|R|+\ell_2$,
where $\ell_2$ is the cardinality of $\set{x\in(0,2)\,|\,\exists j\in R\, ,\, z'_j\in \set{x}\times(0,1)}$.

Unless some $z_i$ coincides with some $z'_j$ (in which case $\mu_{Q,R}$ sends $((z_i)_{i\in Q},(z_j)_{j\in R})$ to $\infty$),
we have that $\mu_{Q,R}((z_i)_{i\in Q},(z_j)_{j\in R})$ lies in $F_S(\M)$, in the interior of a cell
of dimension $|S|+\ell$, where $\ell$ is the cardinality of the union of the two sets considered above:
\[
\begin{split}
 \ell& \!=\!\Big|\!\big\{x\in(0,2)\,|\,\exists i\in R\, ,\, z_i\in \set{x}\times(0,1)\big\}\!\cup\!
 \set{x\in(0,2)\,|\,\exists j\in R\, ,\, z'_j\in \set{x}\times(0,1)}\!\!\Big|\\
 &\leq \ell_1+\ell_2.  
\end{split}
\]
The claim follows. \end{proof}

\begin{remark}\label{factorisation}
A particular instance of the fact that $\mu_{Q,R}$ is cellular is the following. Let $\ft=(\ell,\uP,(),())$
be a tuple on $Q$, where we use the notation from Definition \ref{defn:et}, and let $\ft'=(0,(),\uU,\uV)$.
Then the product cell $e_{\ft}\times e_{\ft'}\subset F_Q(\M)^\infty\wedge F_R(\M)^\infty$ is mapped along
$\mu_{Q,R}$ homeomorphically onto the cell $e_{\ft''}\subset F_S(\M)^\infty$, where $\ft''=(\ell,\uP,\uU,\uV)$;
viceversa, note that each tuple $\ft''$ admits a unique ``factorisation'' as product of two tuples of the form
$(\ell,\uP,(),())$ and $(0,(),\uU,\uV)$.
\end{remark}

Let $f\in\Homeo(\M,\del\M\cup D)$: to understand the behaviour of $\tau_*\circ f_*$ on the cell $e_{\ft''}$,
it suffices to understand its behaviour on the ``factors'' $e_\ft$ and $e_\ft'$, using the fact that
$\mu_{Q,R}$ is a cellular map. At the end of Subsection \ref{subsec:themaptau} we have given a satisfactory
(for our purposes) description of the action of $\tau_*\circ f_*$ on the first factor $e_\ft$. In the next
section we will focus on the second factor, invoking Moriyama's result.

\section{The cellular chain complex of \texorpdfstring{$F_S(\M)^\infty$}{FS(M)infty}}
\label{sec5}

Fix a homeomorphism $f\in\Homeo(\M,\del\M\cup D)$ as in the previous section. In this section
we study the action of $F_S(\tau\circ f)$ on the relative cellular chain complex
\[
\tCh_*^{\cell}(F_S(\M)^\infty)\colon =\Ch^{\cell}_*(F_S(\M)^\infty,\infty),
\]
which is freely generated, as abelian group, by elements $[e_\ft]$ corresponding to
the cells $e_\ft$, which in turn correspond to all tuples on $S$.
We will simplify the notation and write $\tau f_*$ for the induced map of chain complexes
\[
 F_S(\tau\circ f)^\infty_*\colon \tCh_*^{\cell}(F_S(\M)^\infty)\to \tCh_*^{\cell}(F_S(\M)^\infty).
\]
By the discussion in Subsection \ref{subsec:themaptau}, we have that $\tau f_*[e_\ft]=[e_\ft]$
for all tuples $\ft$ of the form $(\ell,\uP,(),())$. In this section we shall focus on the \emph{opposite}
type of tuples, namely those of the form $(0,(),\uU,\uV)$.

\subsection{The one-skeleton \texorpdfstring{$X$}{X} of \texorpdfstring{$\M$}{M}}
\begin{defn}
 \label{defn:X}
 Following Moriyama's notation, we let $X\subset\M$ be the union
 \[
  X=\set{p_0}\cup\bigcup_{i=1}^g\pa{\U_i\cup\V_i}.
 \]
 Note that $X$ is a quotient of the space $2\times[0,1]$.
 We let $\sigma\colon\M\to X$ be the map induced on the quotient by the map $\tilde\sigma\colon[0,2]\times[0,1]\to 2\times[0,1]$
 given by $(x,y)\mapsto (2,y)$.
\end{defn}
 Note that $\sigma$ is a retraction of $\M$ onto $X$, and $\sigma$ is homotopic to the identity of $\M$ through
 self-maps of $\M$ that fix $X$ (and in particular fix $p_0$).

\begin{nota} \label{notationEX}
We denote by $\tCh^E_*(F_S(\M)^\infty)$ the sub-graded abelian group of $\tCh^{\cell}_*(F_S(\M)^\infty)$
spanned by all generators $[e_\ft]$ corresponding to tuples on $S$ of the form $(\ell,\uP,(),())$.

We denote by $\tCh^X_*(F_S(\M)^\infty)$ the sub-graded abelian group of $\tCh^{\cell}_*(F_S(\M)^\infty)$
spanned by all generators $[e_\ft]$ corresponding to tuples on $S$ of the form $(0,(),\uU,\uV)$. 
\end{nota}
Note that $\tCh^X_*(F_S(\M)^\infty)$ is indeed a sub-chain complex of $\tCh^{\cell}_*(F_S(\M)^\infty)$.
More precisely we can consider the subspace $X^S\subset\M^S$, and denote by $\Delta_S(X)$ and $A_S(X)$,
respectively, the intersections $\Delta_S(\M)\cap X^S$ and  $A_S(\M)\cap X^S$;
then the relative cell decomposition of the couple $(\M^S,\Delta_S\cup A_S)$ restricts to a relative cell decomposition
of the couple $(X^S,\Delta_S(X)\cup A_S(X))$, and $\tCh^X_*(F_S(\M)^\infty)$ can be canonically identified with the cellular
chain complex $\Ch^{\cell}_*(X^S,\Delta_S(X)\cup A_S(X))$.
On the contrary, $\tCh^E_*(F_S(\M)^\infty)$ is in general not a sub-chain complex.

The following lemma is essentially \cite[Lemma 4.1]{Moriyama}; we sketch its proof for completeness.
\begin{lem}
The chain complex $\tCh^X_*(F_S(\M)^\infty)$ is concentrated in degree $|S|$, so in particular
it is isomorphic to its $|S|$\textsuperscript{th} homology group
\[
H_{|S|}\pa{\tCh^X_*(F_S(\M)^\infty)}\cong H_{|S|} (X^S,\Delta_S(X)\cup A_S(X)).
\]
The inclusion of couples $(X^S,\Delta_S(X)\cup A_S(X))\subset (\M^S,\Delta_S(\M)\cup A_S(\M))$ is a homotopy equivalence,
in particular it induces an isomorphism on the $|S|$\textsuperscript{th} homology group.
\end{lem}
\begin{proof}
A tuple $\ft$ of length 0 has degree $d(\ft)=|S|$; the degree $d(\ft)$ is the dimension of $e_\ft$, i.e. the degree in
the cellular chain complex $\tCh_*^\cell(F_S(\M)^\infty)$ of the generator $[e_\ft]$. Since 
$\tCh^X_*(F_S(\M)^\infty)$ is spanned by all generators $[e_\ft]$ corresponding to tuples of length 0, the first statement is proved.

For the second statement, note that the map $\sigma\colon \M\to X\subset\M$ is a deformation retraction; since the
deformation preserves $p_0$ at all times, the induced map $\sigma^S\colon \M^S\to X^S$ is a deformation
retraction preserving the subspaces $\Delta_S(\M)$, $A_S(\M)$, $X^S$ as well as the intersections $\Delta_S(X)$ and 
$A_S(X)$. Hence
\[
 \sigma^S\colon (\M^S,\Delta_S(\M)\cup A_S(\M))\to (\M^S,\Delta_S(X)\cup A_S(X))
\]
is an inverse homotopy equivalence to the inclusion of couples
\[
(X^S,\Delta_S(X)\cup A_S(X))\subset (\M^S,\Delta_S(\M)\cup A_S(\M)).
\]
The claim follows. \end{proof}

For $f\in\Homeo(\M,\del\M\cup D)$, note now that $(\tau\circ f)^S\colon\M^S\to\M^S$ restricts to a map $X^S\to X^S$, and the following diagram of couples of spaces
is commutative on the nose
\[
 \begin{tikzcd}
  (X^S,\Delta_S(X)\cup A_S(X))\ar[d,"(\tau\circ f)^S"]\ar[r,"\subset"] & (\M^S,\Delta_S(\M)\cup A_S(\M)) \ar[d,"(\tau\circ f)^S"]\ar[r] & (F_S(\M)^\infty,\infty)\ar[d,"F_S(\tau\circ f)^\infty"]\\
  (X^S,\Delta_S(X)\cup A_S(X))\ar[r,"\subset"] & (\M^S,\Delta_S(\M)\cup A_S(\M)) \ar[r] & (F_S(\M)^\infty,\infty).\\
 \end{tikzcd}
\]
Passing to cellular chain complexes we obtain a strictly commutative diagram
\[
 \begin{tikzcd}
 \tCh^X_*(F_S(\M)^\infty) \ar[d,"\tau f_*"]\ar[r,"\cong"] & H_{|S|}(X^S,\Delta_S(X)\cup A_S(X))\ar[r,"\cong"] \ar[d,"\tau f_*"]&\Mor_S \ar[d,"\tau f_*"]\\
 \tCh^X_*(F_S(\M)^\infty) \ar[r,"\cong"] & H_{|S|}(X^S,\Delta_S(X)\cup A_S(X))\ar[r,"\cong"] &\Mor_S.\\
 \end{tikzcd}
\]
It follows that the action of $\tau f_*$ on the subcomplex $\tCh^X_*(F_S(\M)^\infty)$ of the cellular chain complex
$\tCh_*(F_S(\M)^\infty)$ corresponds to the action of $\tau f_*$ (or just $f_*$) on the homology
group $\Mor_S$. This shows in particular that if $f_1,f_2\in \Homeo(\M,\del\M\cup D)$ are two homeomorphisms,
then the actions of $F_S(\tau\circ f_1)^\infty$ and $F_S(\tau\circ f_2)^\infty$ on the subcomplex
$\tCh^X_*(F_S(\M)^\infty)$ behave well under composition, i.e. we have an equality of maps of chain complexes
(for our purposes, abelian groups sitting in degree $|S|$)
\[
 F_S(\tau\circ f_1)^\infty\circ F_S(\tau\circ f_2)^\infty=F_S(\tau\circ f_1\circ f_2)\colon \tCh^X_*(F_S(\M)^\infty)\to \tCh^X_*(F_S(\M)^\infty).
\]
This last remark could be more easily deduced from the fact that $\tau\circ f_1\circ \tau\circ f_2$
and $\tau\circ f_1\circ f_2$ are homotopic maps $(\M,p_0)\to (\M,p_0)$.
The previous discussion gives the following lemma.
\begin{lem}
 \label{lem:XMoriso}
 For $f$ ranging in $\Homeo(\M,\del\M\cup D)$,
 the cellular chain maps $\tau f_*$ assemble into an action of $\Gamma$
 on the abelian group $\tCh^X_*(F_S(\M)^\infty)$. The $\Gamma$-representations $\tCh^X_*(F_S(\M)^\infty)$ and $\Mor_S$
 are isomorphic.
\end{lem}

\subsection{Cohomological portion of Theorem \ref{thm:main}}
Let $S$ be a finite set. Lemma \ref{lem:mucellular} states that the product map  \[\mu_{Q,R}\colon F_Q(\M)^\infty\wedge F_R(\M)^\infty\to F_S(\M)^\infty\] is cellular and hence it induces a map on cellular chains. By Remark \ref{factorisation}, every cell in $F_S(\M)^\infty$ decomposes uniquely as a product of a cell of the form $(0,(),\uU,\uV)$ and a cell of the form $(\ell,\uP,(),())$. In the language of Notation \ref{notationEX}, this gives an isomorphism of graded abelian groups
\[
 \tCh_*^{\cell}(F_S(\M)^\infty)\cong \bigoplus_{S=Q\sqcup R} \tCh^E_*(F_Q(\M)^\infty)\otimes \tCh^X_*(F_R(\M)^\infty).
\]
For $f\in\Homeo(\M,\del\M\cup D)$, the action of $\tau f_*$ on the left hand side corresponds to the direct sum
of the tensor product of the identity action on each factor $\tCh^E_*(F_Q(\M)^\infty)$ with the Moriyama action
of $\tau f_*$ on $\tCh^X_*(F_R(\M)^\infty)$. By the discussion of the previous subsection we obtain the following proposition.
\begin{prop}
\label{prop:main}
Let $S$ be a finite set. There is an isomorphism of graded $\Gamma$-representations
\[
 \tCh_*^{\cell}(F_S(\M)^\infty)\cong \bigoplus_{S=Q\sqcup R} \tCh^E_*(F_Q(\M)^\infty)\otimes\Mor_R.
\] 
Here $Q$ and $R$ range among all couples of (possibly empty) subsets of $S$ whose disjoint union is $S$. The action
of $\Gamma$ on the right hand side preserves direct summands, is diagonal on each tensor product, is the identity
action on each tensor factor $\tCh^E_*(F_Q(\M)^\infty)$ and is the Moriyama action on each tensor factor $\Mor_R$.
\end{prop}

In particular, if $f\in\Homeo(\M,\del\M\cup D)$ represents an element in $\cJ(i)$,
then by Moriyama's result (Theorem \ref{thm:Moriyama})
we have that $\tau f_*$ restricts to the identity on those direct summands in the formula
of Proposition \ref{prop:main} that depend on a subset $R\subset S$ with $|R|\le i$. The direct sum
\[
  \bigoplus_{{S=Q\sqcup R}\, ,\, {|R|\le i}} \tCh^E_*(F_Q(\M)^\infty)\otimes\Mor_R
\]
contains the degree $\ge 2|S|-i$ part of the chain complex $ \tCh_*^{\cell}(F_S(\M)^\infty)$. It follows that
$\tau f_*$, hence $f_*$, are the identity on the homology group
\[
 H_{2|S|-i} \tCh_*^{\cell}(F_S(\M)^\infty).
\]
By Poincare-Lefschetz duality and the comparison between cellular and singular homology (see Equation \ref{DiagramPLD}), this implies that $f_*$ acts as the identity on $H^i(F_S(\M))$,
which is the cohomological portion of Theorem \ref{thm:main}.

In the case $i=|S|$ we can say a little more: for all $f\in \Homeo(\M,\del\M\cup U)$ representing an element
in $\cJ(|S|)$, we have that $\tau f_*$ is the identity on the entire chain complex $\tCh_*^{\cell}(F_S(\M)^\infty)$;
using the quasi-isomorphisms of Subsection \ref{subsec:naturalquasiiso}, we obtain that
\[
f_*\colon \Ch^*_{\sing}(F_S(\M))\to\Ch^*_{\sing}(F_S(\M))
\]
is chain homotopic to the identity.
In other words, $\cJ(|S|)$ acts homotopy-trivially on the cochain complex $\Ch^*_{\sing}(F_S(\M))$.

\subsection{Homological portion of Theorem \ref{thm:main}}
By Poincare-Lefschetz duality, and the comparison between cellular and singular homology,
$\Ch_*^{\sing}(F_S(\M))$ can be replaced by the reduced cellular \emph{cochain} complex $\Ch^*_{\cell}(F_S(\M)^\infty,\infty)$,
which is just the dual of the chain complex
$\Ch_*^{\cell}(F_S(\M)^\infty,\infty)$: note indeed that the latter chain complex is
finitely generated and free abelian. The rest of the arguments are replaced by their duals. We therefore deduce the homological portion of Theorem \ref{thm:main}.

\subsection{An analogue of Theorem \ref{thm:main} for labeled configuration spaces}
The same strategy of proof generalises to the setting of labeled configuration spaces.
Let $Y$ be a topological space with an action of the group $\fS_n$. We consider the space
$F_n(\mathring{\M})\times_{\fS_n}Y$, i.e. the quotient of $F_n(\mathring{\M})\times Y$ by the diagonal $\fS_n$-action.
The most prominent example is when $Y=Z^n$ for some topological space $Z$, with $\fS_n$ acting on $Y$ by
permuting the $n$ coordinates of $Z^n$: in this case $F_n(\mathring{\M})\times_{\fS_n}Y$ is also known as the
\emph{unordered configuration spaces of $n$ points in $\M$ with labels in $Z$} and denoted $C_n(\M;Z)$.

The topological group $\Homeo(\M,\del\M)$ acts on $F_n(\mathring{\M})\times Y$: naturally on the first factor,
trivially on the second. This action induces an action of $\Homeo(\M,\del\M)$ on $F_n(\mathring{\M})\times_{\fS_n}Y$;
the latter descends to a homotopy action of $\Gamma$ on the chain complex
$\Ch_*^{\sing}(F_n(\mathring{\M})\times_{\fS_n}Y)$.
We will prove the following corollary.
\begin{cor}
 \label{cor:main}
The action of $\cJ(n)\subset\Gamma$ on  $\Ch_*^{\sing}(F_n(\mathring{\M})\times_{\fS_n}Y)$ is homotopy-trivial.
In particular $\cJ(n)$ acts trivially on $H_*(F_n(\mathring{\M})\times_{\fS_n}Y)$.
\end{cor}
\begin{proof}
 We first observe that $\Ch_*^{\sing}(F_n(\mathring{\M}))$ can be regarded as a $\fS_n$-equivariant chain complex.
 All arguments leading to Theorem \ref{thm:main} hold in the setting of $\fS_n$-equivariant spaces and chain complexes:
 in particular if $f\in\Homeo(\M,\del\M)$ represents a mapping class in $\cJ(n)\subset\Gamma$, then
 $F_n(f)_*\colon \Ch_*^{\sing}(F_n(\mathring{\M}))\to \Ch_*^{\sing}(F_n(\mathring{\M}))$ is $\fS_n$-equivariantly chain homotopic to the identity of
 $\Ch_*^{\sing}(F_n(\mathring{\M}))$.
 
 Consider now the Eilenberg-Zilber map
 \[
  \mathrm{EZ}\colon \Ch_*^{\sing}(F_n(\mathring{\M}))\otimes_{\fS_n}\Ch_*^{\sing}(Y) \to \Ch_*^{\sing}(F_n(\mathring{\M})\times_{\fS_n}Y),
 \]
 it is a quasi-isomorphism of chain complexes; moreover if $f\in\Homeo(\M,\del\M)$
 is any homeomorphism, the action of $(F_n(f)\times_{\fS_n}\Id_Y)_*$ on the right hand side is compatible with the action
 of $F_n(f)_*\otimes_{\fS_n}\Id_{\Ch_*^{\sing}(Y) }$ on the left hand side, i.e. the following diagram of chain complexes
 commutes (on the nose)
 \[
  \begin{tikzcd}
   \Ch_*^{\sing}(F_n(\mathring{\M}))\otimes_{\fS_n}\Ch_*^{\sing}(Y)\ar[d,"F_n(f)_*\otimes_{\fS_n}\Id_{\Ch_*^{\sing}(Y) }"]\ar[r,"\mathrm{EZ}"]
   & \Ch_*^{\sing}(F_n(\mathring{\M})\times_{\fS_n}Y)\ar[d,"(F_n(f)\times_{\fS_n}\Id_Y)_*"]\\
   \Ch_*^{\sing}(F_n(\mathring{\M}))\otimes_{\fS_n}\Ch_*^{\sing}(Y)\ar[r,"\mathrm{EZ}"]
   & \Ch_*^{\sing}(F_n(\mathring{\M})\times_{\fS_n}Y).
  \end{tikzcd}
 \]
 Suppose now that $f$ represents a mapping class in $\cJ(n)\subset\Gamma$: then the left vertical map is
 chain homotopic to the identity of $\Ch_*^{\sing}(F_n(\mathring{\M}))\otimes_{\fS_n}\Ch_*^{\sing}(Y)$,
 because $F_n(f)_*$ is \emph{$\fS_n$-equivariantly} chain homotopic to the identity of
 $\Ch_*^{\sing}(F_n(\mathring{\M}))$. Since $\mathrm{EZ}$ is a quasi-isomorphism of chain complexes, it follows
 that the right vertical map is also chain homotopic to the identity of $\Ch_*^{\sing}(F_n(\mathring{\M})\times_{\fS_n}Y)$. 
 \end{proof}

\section{Sharpness of Theorem \ref{thm:main} and examples}
\label{sec:sharpness}

In this section we present some examples testing the sharpness of our main theorem. We first
show that $\cJ(2)$ acts non-trivially on $H_3(F_3(\M))$ if $\M$ has genus $\ge3$ and, with a finer
argument, if $\M$ has genus 2.

We then observe that, if $\M$ has genus $\ge2$, the subgroup of $\Gamma=\Gamma(\M,\del\M)$
generated by the boundary Dehn twist is not contained in $\cJ(3)$, and nevertheless this subgroup
acts trivially on $H_i(F_n(\M))$ for all $n,i\ge0$. In particular, $\cJ(i)$ is not generally equal to the kernel of the action of $\Gamma$ on $H_i(F_n(\M))$.

\subsection{\texorpdfstring{$\cJ(2)$}{J(2)} does not act trivially on \texorpdfstring{$H_3(F_3(\mathring{\M}))$}{H3(F3(M))} for \texorpdfstring{$g\ge3$}{g>=3}}
\begin{prop}
\label{H3F3}
Let $g \geq 3$. The group $\cJ(2)$ acts nontrivially on $H_3(F_3(\mathring{\M}))$. 
\end{prop} 

\begin{proof}
Consider a subsurface $\Sigma_{3, 1} \subseteq \M$.  Let $\gamma$ be the simple closed curve shown in Figure \ref{3HandleSurface}, and let $T_{\gamma}$ denote the Dehn twist about $\gamma$.  
Since $\gamma$ is a separating curve, $T_{\gamma} \in \cJ(2)$ by work of Johnson \cite{Johnson80,JohnsonII}.

\begin{figure}[!ht]    \centering
\labellist
\Large \hair 0pt
\pinlabel { {\color{orange} $\gamma$}} [r] at 500 510 
\endlabellist
\includegraphics[scale=.2]{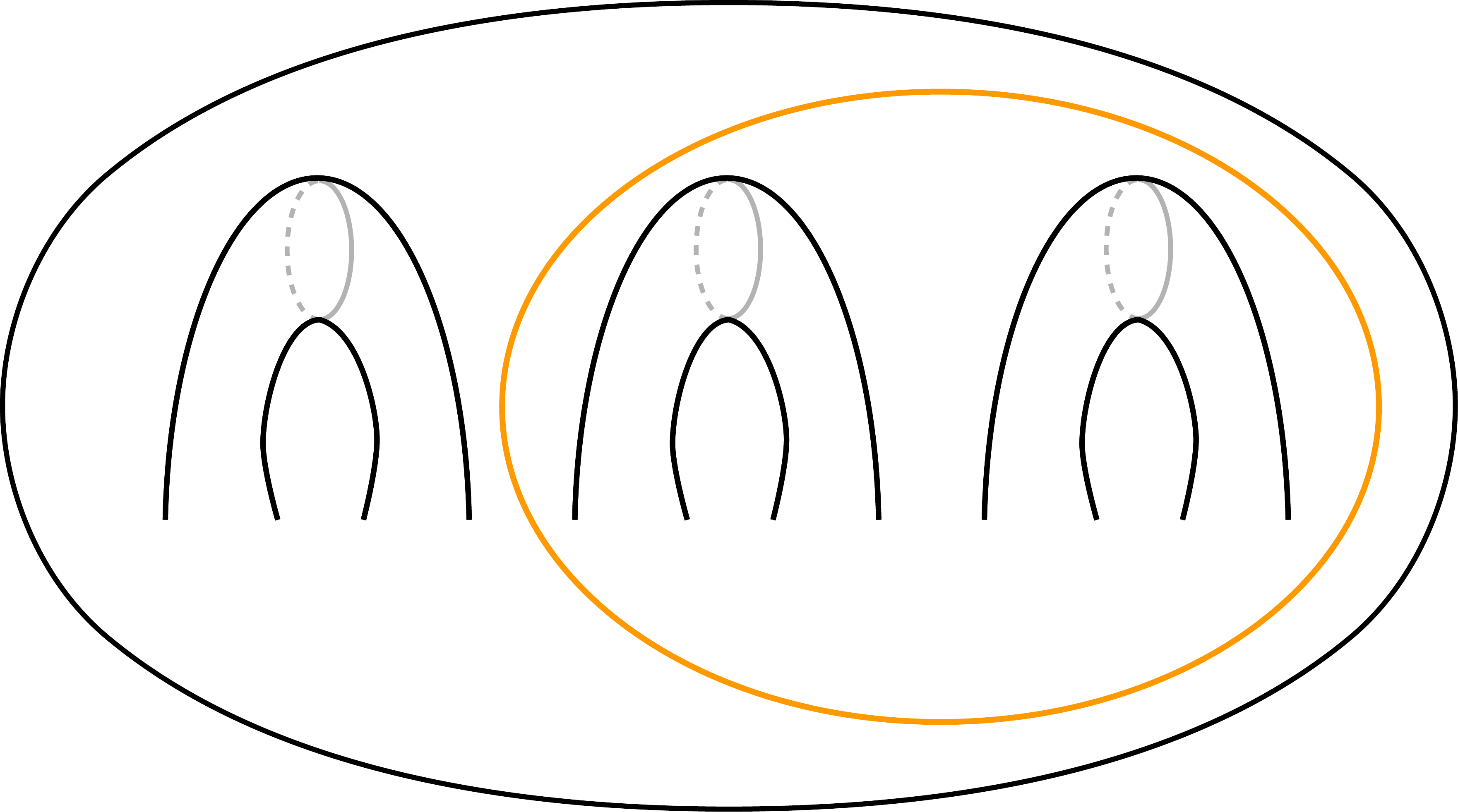} \\ \quad \\ 
\caption{The surface $\Sigma_{3,1}$ and the curve $\gamma$.}
\label{3HandleSurface}
\end{figure}  

Consider the disjoint subsurfaces $\Sigma' \cong \Sigma_{2,1}$ and $\Sigma'' \cong \Sigma_{1,1}$ shown in Figure \ref{3HandleSurfaceConfigDecomp}. 
\begin{figure}[!ht]    \centering
\labellist
\Large \hair 0pt
\pinlabel { { $\Sigma''$}} [l] at 700 130 
\pinlabel { { $\Sigma'$}} [l] at 250 120
\endlabellist
\includegraphics[scale=.15]{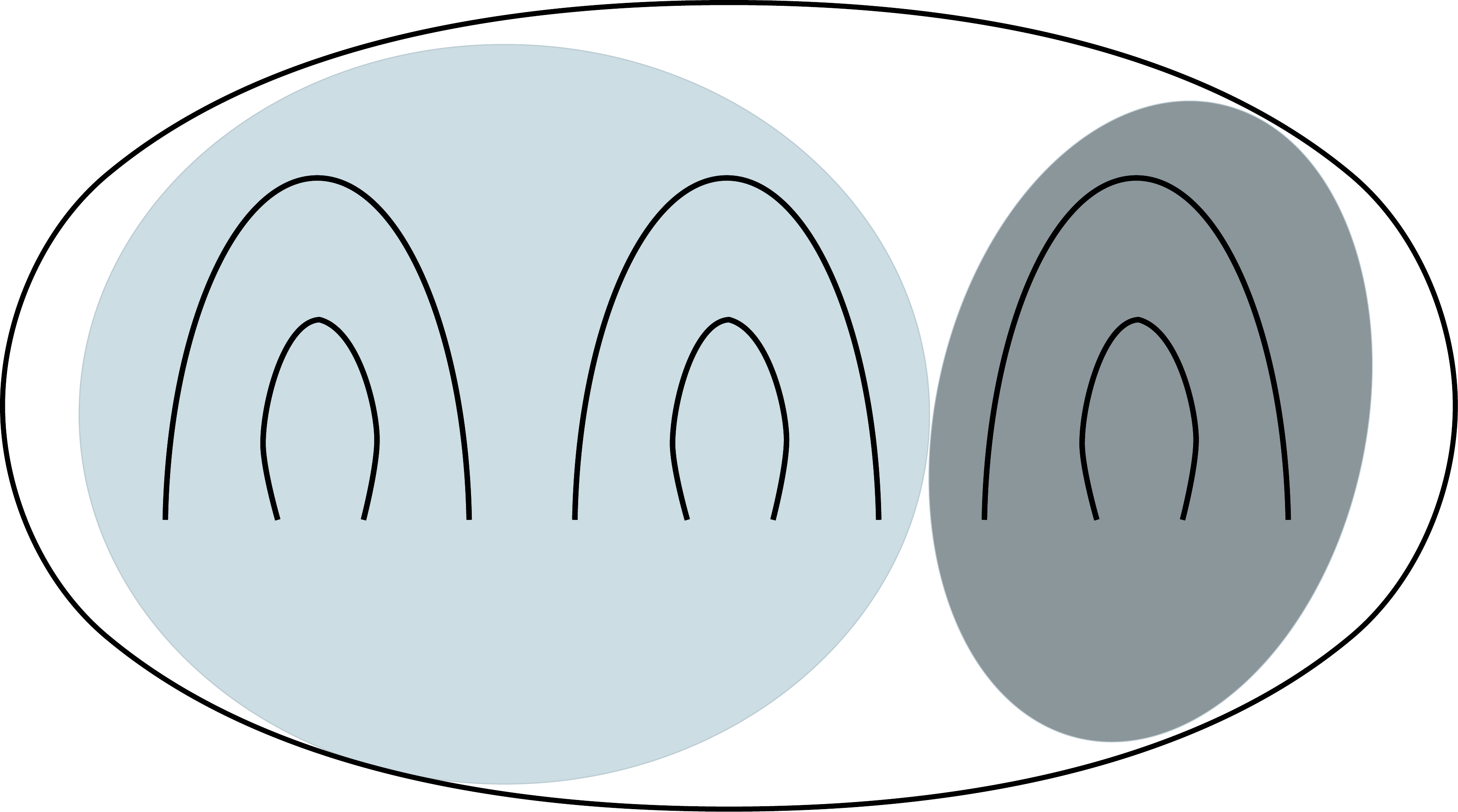} \\ \quad \\ 
\caption{The disjoint subsurfaces $\Sigma'$ and $\Sigma''$.}
\label{3HandleSurfaceConfigDecomp}
\end{figure}
We will construct chains $a \in C_2(F_{\{1,2\}}(\mathring{\Sigma}'))$ and $b \in C_1(F_{\{3\}}(\mathring{\Sigma}''))$ and write $a \bullet b$ to denote the image under the natural map
\begin{align*}
C_2(F_{\{1,2\}}(\mathring{\Sigma}')) \otimes C_1(F_{\{3\}}(\mathring{\Sigma}'')) & \longrightarrow C_3( (F_{\{1,2,3\}}(\mathring{\Sigma}_{3,1})),\\
 a\otimes b & \longmapsto a \bullet b .
\end{align*} 

The chain $b$ corresponds to the loop in $F_{\{3\}}(\mathring{\Sigma}'')$ in which particle 3 traverses the curve $\delta$ in Figure \ref{3HandleSurfaceConfigClass}. We construct the chain $a$ as follows.
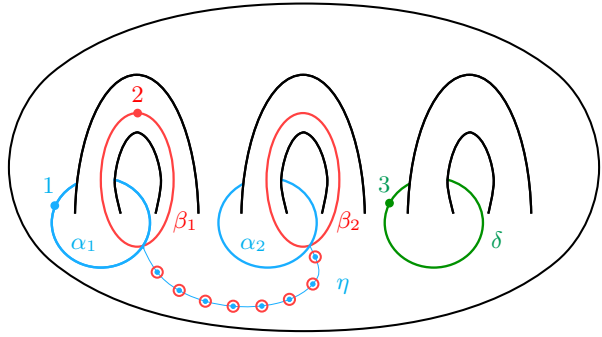
\begin{figure}[!ht]    \centering
\labellist
\Large \hair 0pt
\pinlabel { \small {\color{ForestGreen} $\delta$}} [l] at 810 160 
\pinlabel {\small {\color{ForestGreen} $3$}} [r] at 655 255 
\pinlabel {\small {\color{red} $2$}} [u] at 222 410 
\pinlabel {\small {\color{ProcessBlue} $1$}} [r] at 80 252 
\pinlabel { \small {\color{ProcessBlue} $\alpha_1$}} [l] at 90 150 
\pinlabel { \small {\color{ProcessBlue} $\alpha_2$}} [l] at 380 150
\pinlabel { \small {\color{red} $\beta_1$}} [l] at 265 190 
\pinlabel { \small {\color{red} $\beta_2$}} [l] at 545 190
\pinlabel { \small {\color{ProcessBlue} $\eta$}} [l] at 545 80
\endlabellist
\includegraphics[scale=.23]{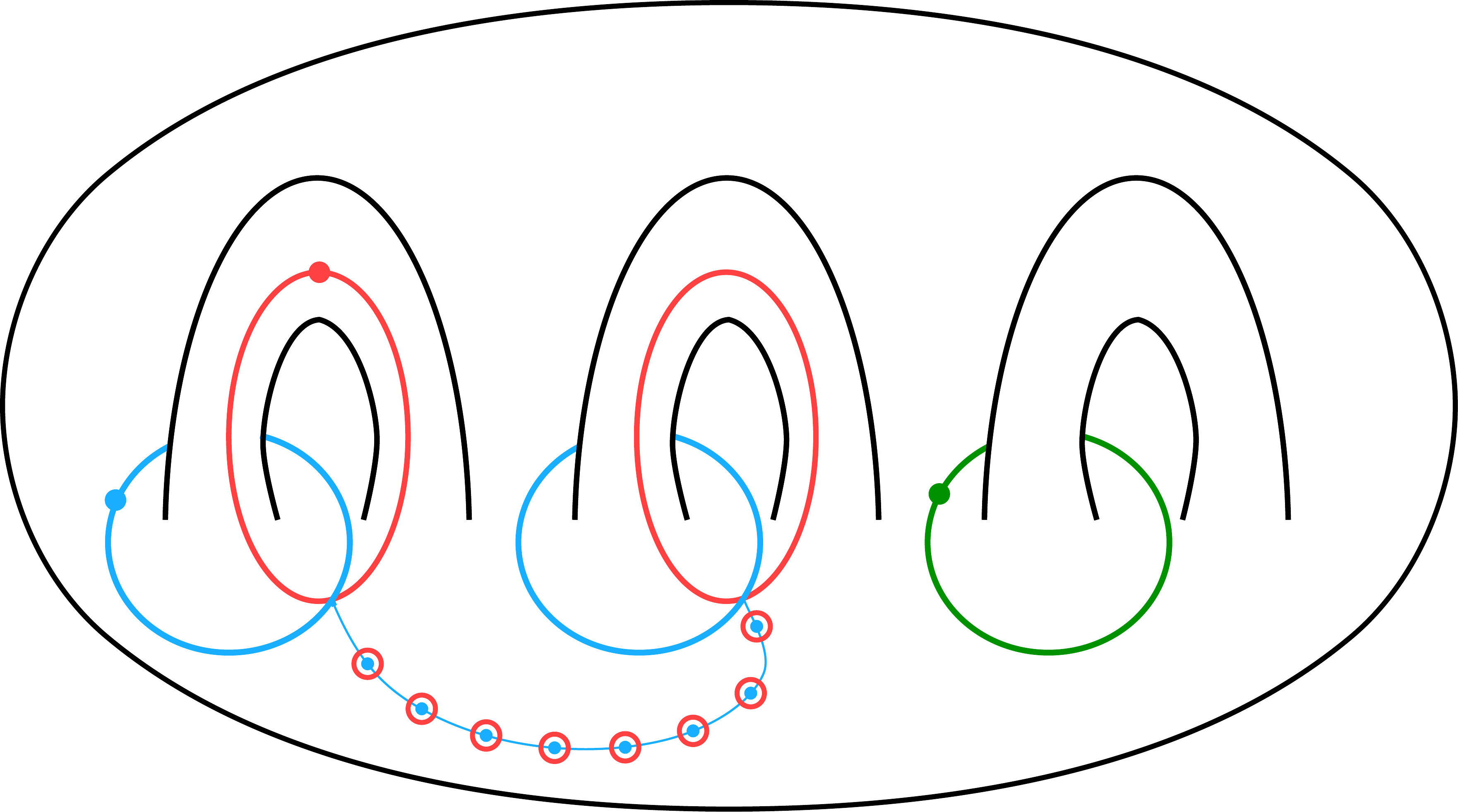} \\ \quad \\
\caption{The chain $a \bullet b$.}
\label{3HandleSurfaceConfigClass}
\end{figure}   

Consider the simple closed curves $\alpha_1$, $\beta_1$, $\alpha_2$, $\beta_2$ in Figure \ref{3HandleSurfaceConfigClass}. 
The curves $\alpha_1$ and $\beta_1$ intersect at a single point $x_1$, and $\alpha_2$ and $\beta_2$ intersect at a single point $x_2$. We define a embedded torus
\[
\sigma_1'\colon \bfT^2=S^1\times S^1 \to (\mathring{\Sigma} ')^2
\]
 by letting particle 1 traverse $\alpha_1$ and particle 2 traverse $\beta_1$. Delete an open ball $B$ from our torus around the preimage $x\in\bfT^2$ of $(x_1, x_1)\in(\mathring{\Sigma}')^2$. 
The restriction of $\sigma_1'$ to  $\bfT^2\setminus B \cong \Sigma_{1,1}$ is then an embedding into the configuration space 
\[
\sigma_1\colon \bfT^2\setminus B \to F_{\{1,2\}}(\mathring{\Sigma}').
\]
The restriction $\sigma_1|_{\partial B}$ to the boundary of $B$ takes values in the configuration
space $F_{\{1,2\}}(D_{x_1})$, where $D_{x_1}\subset\mathring{\Sigma}'$ is a small disc around $x_1$;
more precisely, the restricted map $\sigma_1|_{\partial B}\colon\partial B\to F_{\{1,2\}}(D_{x_1})$
describes, up to homotopy in $F_{\{1,2\}}(D_{x_1})$ a \emph{planetary system}, i.e. a curve of configurations of two particles labelled 1 and 2 in a disc, in which particle 2 spins once clockwise around particle 1.

We repeat the construction for the curves $\alpha_2$ and $\beta_2$, to produce a second embedding  
\[
\sigma_2: \bfT^2\setminus B \to F_{\{1,2\}}(\mathring{\Sigma}'),
\]
and again the restriction $\sigma_2|_{\partial B}$ gives a planetary system in $F_{\{1,2\}}(D_{x_2})$
for a small disc $D_{x_2}\subset\mathring{\Sigma}'$ around $x_2$.

Our chain $a$ corresponds to an embedded surface of genus 2
\[
\sigma\colon\Sigma_2=(\bfT^2\setminus B)\sqcup_{\del B} (\del B\times[0,1])\sqcup_{\del B}(\bfT^2\setminus B)
\to F_2(\mathring{\Sigma}')
\]
obtained by extending the maps $\sigma_1$ and $\sigma_2$ over a cylinder glued by its boundary components homeomorphically to the boundary components of two copies of $T\setminus B$.
The extension depends on a choice of an arc $\eta$ in $\mathrm{\Sigma}'$ connecting $x_1$ with $x_2$,
yielding an isotopy between $\sigma_1|_{\partial B}$ and $\sigma_2|_{\partial B}$.
A schematic of the resulting chain $a$ is shown in Figure \ref{3HandleSurfaceConfigClass}.

We will prove that $T_{\gamma}( a \bullet b)$ is not homologous to $a \bullet b$. The difference $$d = T_{\gamma}( a \bullet b) - (a \bullet b)$$ is homologous to the 3-torus shown in Figure \ref{3HandleSurfaceConfigClassDifference}, consisting of particle 3 traversing the loop $\delta$ and the particle 1-2 planetary system  traversing the loop $\gamma$. 

\begin{figure}[!ht]    \centering
\labellist
\Large \hair 0pt
\pinlabel { \small {\color{ForestGreen} $\delta$}} [l] at 810 160 
\pinlabel { \small {\color{ProcessBlue} $\gamma$}} [r] at 500 510 
\endlabellist
\includegraphics[scale=.23]{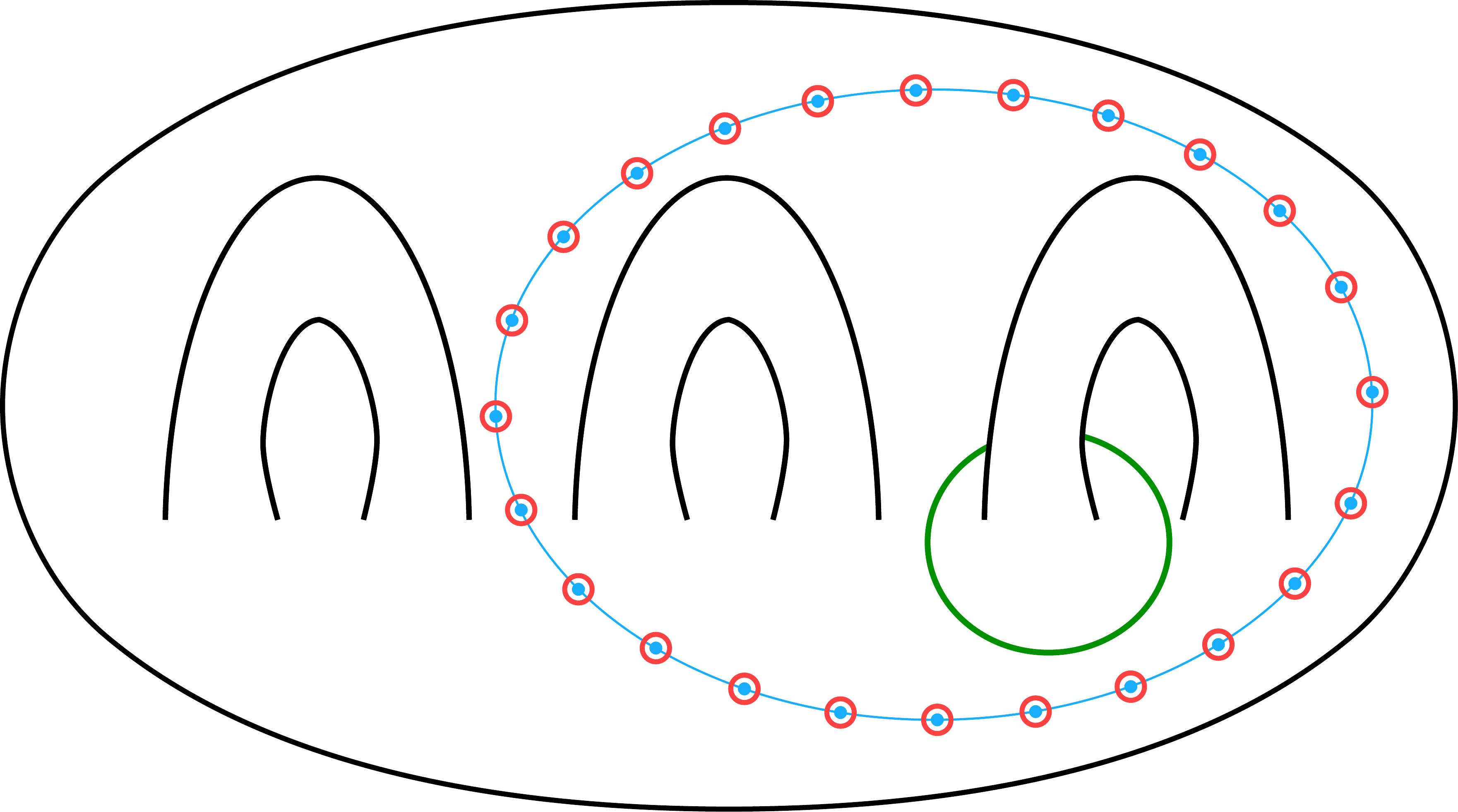} \\ \quad \\ 
\caption{The chain $d = T_{\gamma}( a \bullet b) - (a \bullet b)$}
\label{3HandleSurfaceConfigClassDifference}
\end{figure}  

To verify that $d$ is nonzero in homology, consider the oriented, properly embedded arc $\omega\subset\mathring{\Sigma}_{3,1}$ shown in Figure \ref{3HandleSurfaceConfigClassDifferenceIntersection}. Consider the properly embedded sub-3-manifold of $F_3(\mathring{\Sigma}_{3,1})$ consisting of all configurations in which particles 1, 2, and 3 lie on $\omega$ in this order. This oriented submanifold intersects $d$ transversely once. 

\begin{figure}[!ht]    \centering
\labellist
\Large \hair 0pt
\pinlabel { \small {\color{ForestGreen} $\delta$}} [l] at 750 160 
\pinlabel { \small {\color{ProcessBlue} $\gamma$}} [r] at 500 510 
\pinlabel { \small {\color{Gray} $\omega$}} [r] at 890 180  
\endlabellist
\includegraphics[scale=.23]{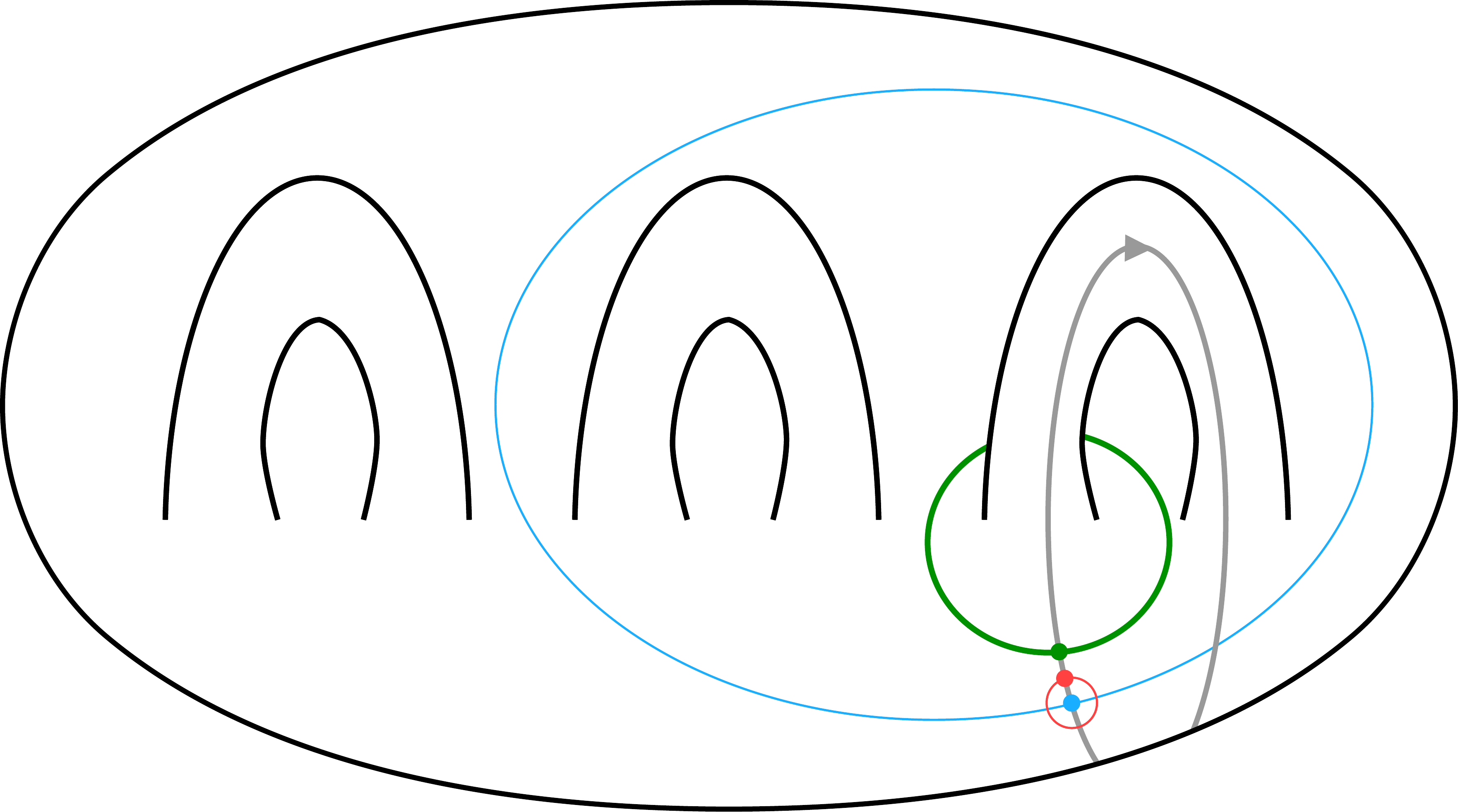} \\ \quad \\ 
\caption{The chain $d = T_{\gamma}( a \bullet b) - (a \bullet b)$}
\label{3HandleSurfaceConfigClassDifferenceIntersection}
\end{figure}  
We conclude that $d$ is nonzero in homology, and hence that $T_{\gamma}( a \bullet b)$ is not homologous to $(a \bullet b)$.  \end{proof} 

\subsection{\texorpdfstring{$\cJ(2)$}{J(2)} does not act trivially on \texorpdfstring{$H_3(F_3(\mathring{\M}))$}{H3(F3(M))} for \texorpdfstring{$g=2$}{g=2}}
We refine the previous argument in the genus 2 case.
\begin{prop}
\label{H3F3g2}
The group $\cJ(2)$ acts nontrivially on $H_3(F_3(\mathring{\M}))$ for $\M$ of genus 2.
\end{prop}

Consider the surface $\Sigma_{2,1}$ and the separating curve $\gamma$ shown in  Figure \ref{2HandleSurface}. 

\begin{figure}[!ht]    \centering
\labellist
\Large \hair 0pt
\pinlabel { {\color{orange} $\gamma$}} [r] at 530 510 
\endlabellist
\includegraphics[scale=.16]{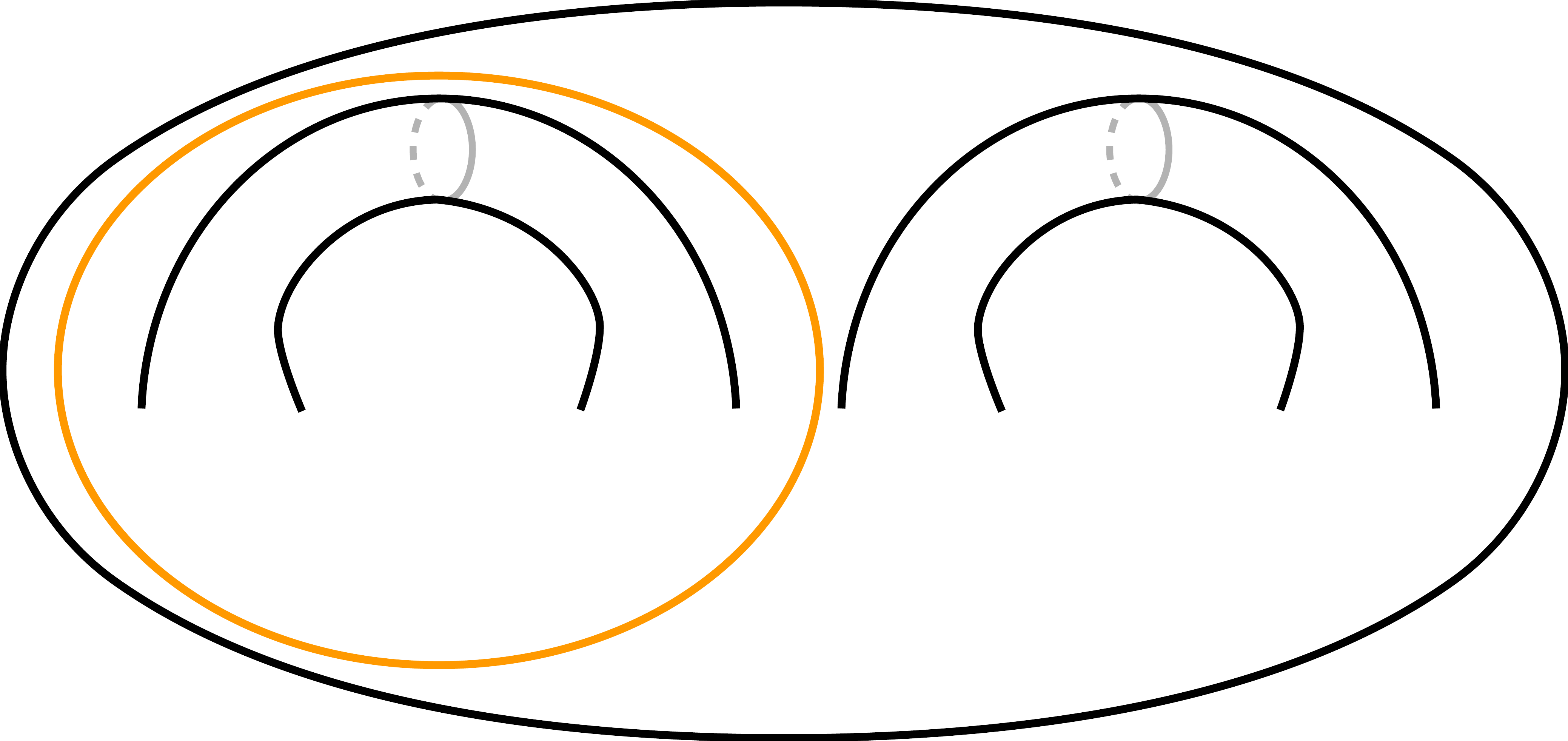} \\ \quad \\ 
\caption{The surface $\Sigma_{2,1}$ and the curve $\gamma$.}
\label{2HandleSurface}
\end{figure}  

We construct a class $a \in H_3( F_3(\mathring{\Sigma}_{2,1}))$, as follows.  We first embed a 3-torus $\bfT^3 = (S^1_{(1)} \times S^1_{(2)} \times S^1_{(3)})$ in $(\mathring{\Sigma}_{2,1})^3$ as in the first image in Figure \ref{2HandleSurfaceClass}.

\begin{figure}[!ht]    \centering
\labellist
\Large \hair 0pt
\pinlabel {\small {\color{red} $1$}} [u] at 335 1545 
\pinlabel {\small {\color{ForestGreen} $3$}} [r] at 390 1660
\pinlabel {\small {\color{ProcessBlue} $2$}} [r] at 320 1660 
\pinlabel {\small {\color{red} $1$}} [u] at 870 870 
\pinlabel {\small {\color{ForestGreen} $3$}} [r] at 920 985
\pinlabel {\small {\color{ProcessBlue} $2$}} [r] at 320 985 
\pinlabel {\small {\color{ProcessBlue} $\alpha$}} [r] at 320 850 
\pinlabel {\small {\color{red} $1$}} [u] at 870 190 
\pinlabel {\small {\color{ForestGreen} $3$}} [r] at 390 310
\pinlabel {\small {\color{ProcessBlue} $2$}} [r] at 850 310
\pinlabel {\huge {\color{black} $+$}} [r] at -50 965
\pinlabel {\huge {\color{black} $+$}} [r] at -50 285
\endlabellist
\includegraphics[scale=.16]{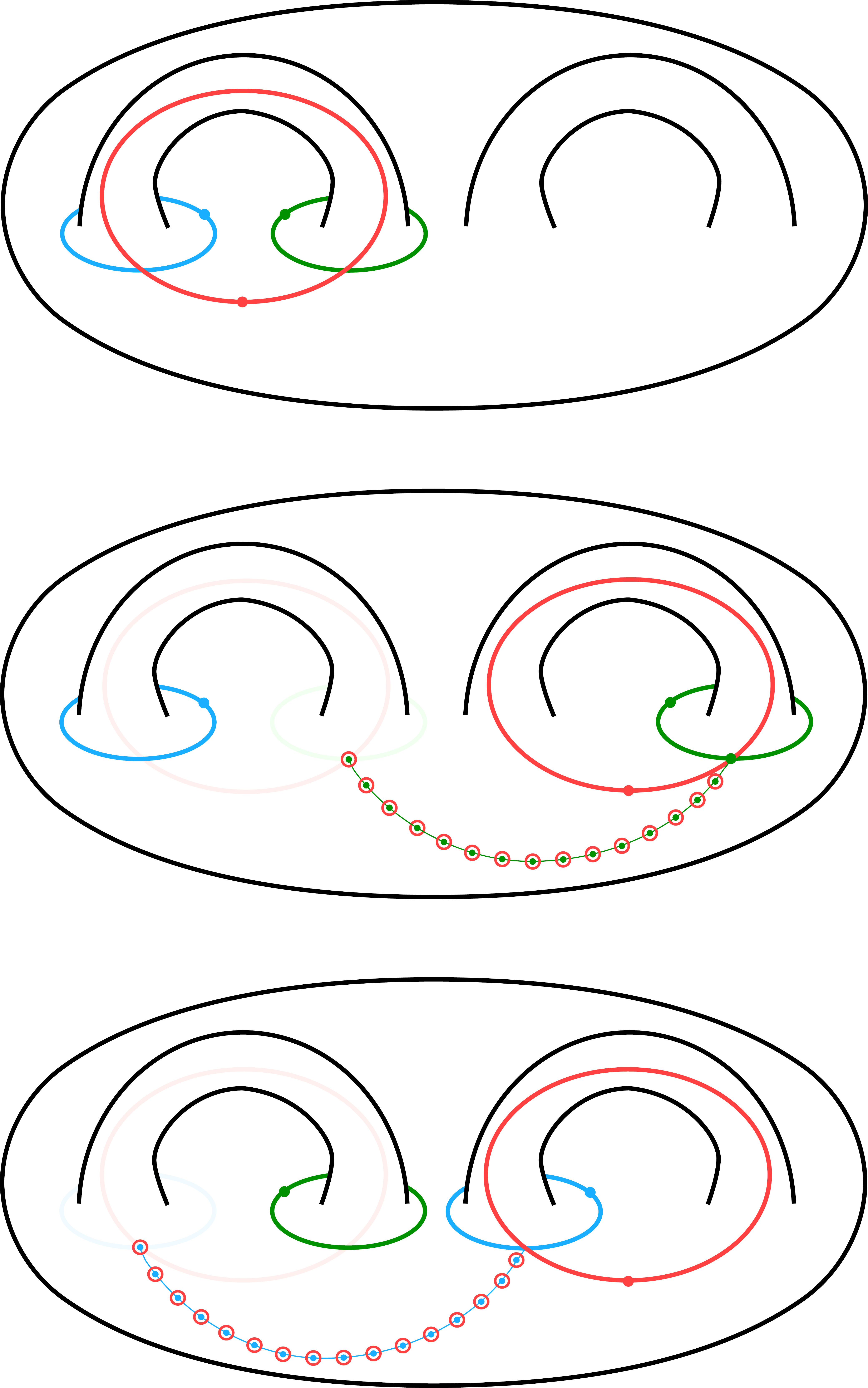} \\ \quad \\ 
\caption{The class $a \in H_3( F_3(\mathring{\Sigma}_{2,1}))$.}
\label{2HandleSurfaceClass}
\end{figure}  

Consider the embedded 2-torus $\bfT^{\{1,2\}}=S^1_{(1)}\times S^1_{(2)}\subset\bfT^3$
parameterised by the particles $1$ and $2$. As in the proof of Proposition \ref{H3F3}, we delete a small ball $B_{12}\subset\bfT^{\{1,2\}}$ around the collision point of the two particles. Correspondingly, we delete a solid torus $\bfT_{12} = B_{12} \times S^1_{(3)}$ from $\bfT^3$. In the same 
way we remove the collision point of particles
$1$ and $3$. The result is an embedded $3$-manifold with boundary
\[
\sigma': \bfT^3 \setminus (\bfT_{12} \sqcup \bfT_{13}) \quad \hookrightarrow \quad F_3(\mathring{\Sigma}_{2,1})
\]
defined by the complement of two disjoint solid tori in a 3-torus.

To construct the class $a$, we will extend $\sigma'$ to a closed 3-manifold obtained by glueing
two other 3-manifolds with boundary along $\del \bfT_{12}$ and $\del\bfT_{13}$.
These are illustrated in the second and third image in Figure \ref{2HandleSurfaceClass}, respectively, using the conventions of Figure \ref{3HandleSurfaceConfigClass}.  The second image shows the product of the circle $S^1_{(2)}$ (parameterised by particle 2) with a surface $\Sigma^{13}\cong\Sigma_{1,1}$ (parameterised by particles 1 and 3). The embedding $\sigma^{13}\colon\Sigma\to F_{\{1,3\}}(\mathring{\Sigma}_{2,1})$ is defined as follows: first, we regard $\Sigma^{13}$ as the union
\[
\Sigma^{13}=(\bfT^{\{1,3\}}\setminus B_{13})\sqcup_{\del B_{13}}\del B_{13}\times[0,1];
\]
we use the cylinder $\del B_{13}\times[0,1]$
to isotope the restriction $\sigma'|_{\partial B_{13}}$ to a map with values in
a small disc inside the rightmost handle $\Sigma''\subset\Sigma_{2,1}$;
then we use $(\bfT^{\{1,3\}}\setminus B_{13})$ to close off the planetary system described by particles
1 and 3, letting again particles 1 and 3 trace the loops shown in Figure \ref{2HandleSurfaceClass}.
We can assume that no configuration hit by the embedding $\sigma^{13}$ touches with some particle
the curve $\alpha$ in the middle image of Figure \ref{2HandleSurfaceClass}: therefore we can cross
$\sigma$ with the map $S^1_{(2)}\to F_{\{2\}}(\mathrm{\Sigma}_{2,1})$ in which we let particle
2 traverse the curve $\alpha$, obtaining finally an embedding $S^1_{(2)}\times \Sigma^{13}\to F_3(\mathring{\Sigma}_{2,1})$ extending on the boundary the map $\sigma'|_{\del\bfT_{13}}$.

The third image is analogous. 

Let $T_{\gamma}$ be the Dehn twist about $\gamma$. We will show that $T_{\gamma}$ acts nontrivially on the class $a$ by verifying that the difference $d = T_{\gamma}(a) - a$ is nonzero. This class is shown in Figure \ref{2HandleSurfaceTwistDifference}, as the sum of two homology classes represented by embedded 3-tori in $F_3(\mathrm{\Sigma}_{2,1})$.

\begin{figure}[!ht]    \centering
\labellist
\Large \hair 0pt
\pinlabel {\huge  {\color{black} $+$}} [r] at 1295 285  
\endlabellist
\includegraphics[scale=.145]{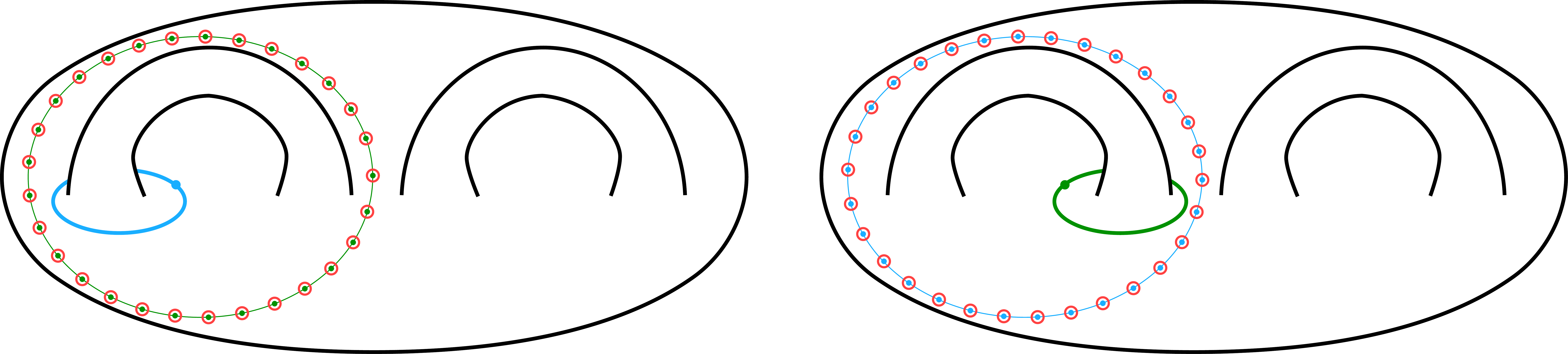} \\ \quad \\ 
\caption{The difference $T_{\gamma}(a) - a$ as a sum of chains}
\label{2HandleSurfaceTwistDifference}
\end{figure}  

 We verify that $d$ is nonzero in homology by the same method as in the proof of Proposition \ref{H3F3}. Consider the 3-manifold of all configurations with particles 1, 2, and 3 lying in this order on the directed arc $\omega$ in Figure \ref{2HandleSurfaceIntersection}.

\begin{figure}[!ht]    \centering
\labellist
\Large \hair 0pt
\pinlabel {\small {\color{gray} $\omega$}} [r] at 500 145  
\endlabellist
\includegraphics[scale=.16]{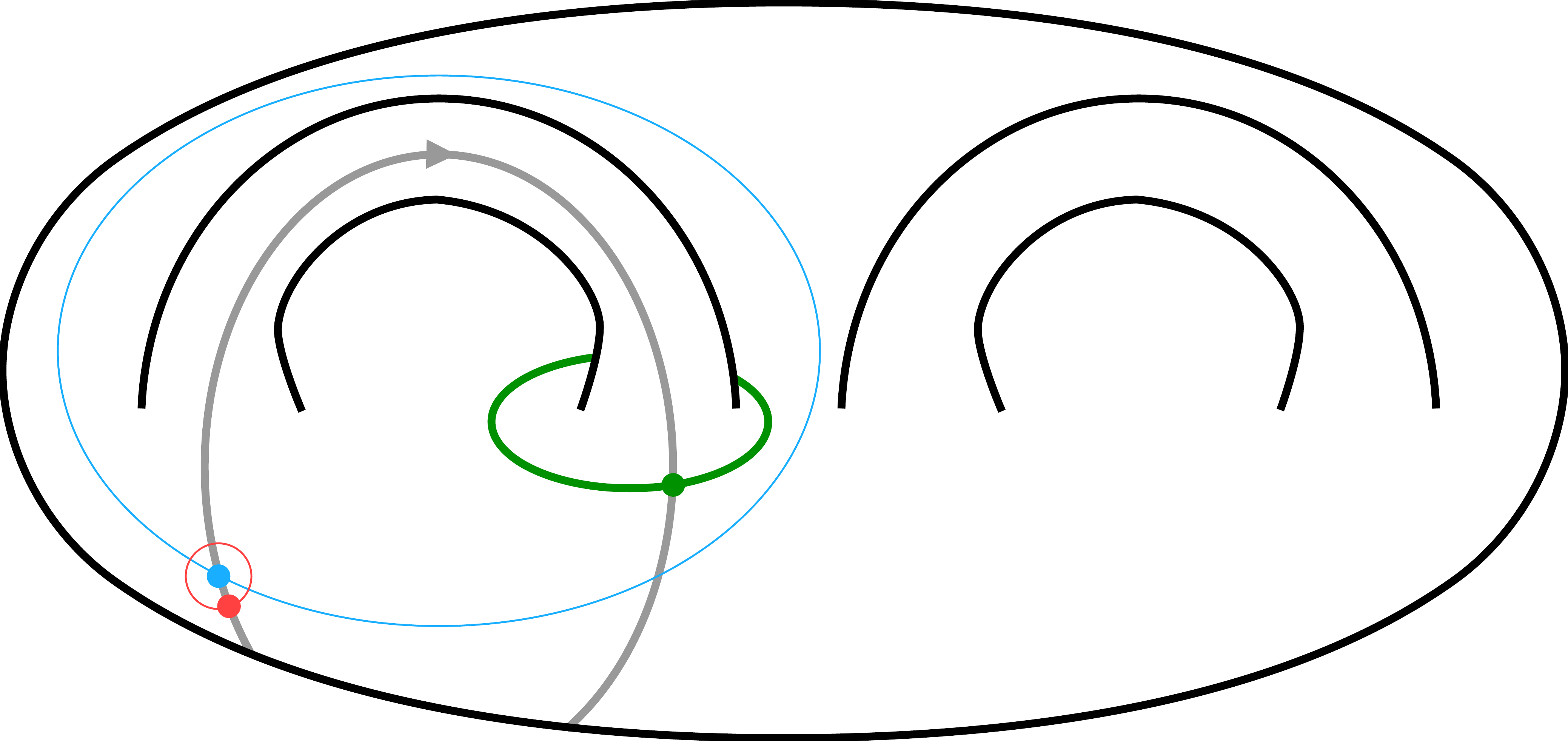} \\ \quad \\ 
\caption{The transverse intersection of sub-3-manifolds}
\label{2HandleSurfaceIntersection}
\end{figure}  

This 3-manifold intersects our representative of $d$ once transversely. We conclude that $d$ is nonzero, as claimed. 

\subsection{\texorpdfstring{$\cJ(i)$}{J(i)} is not the entire kernel of the action on
\texorpdfstring{$H_i(F_n(\mathring{\M}))$}{Hi(Fn(M))} for \texorpdfstring{$i\ge3,g\ge2$}{i>=3,g>=2}}
\label{subsec:bDehnJ2}
The Dehn twist $T_{\del\M}\in\Gamma$ generates, for $g\ge1$, an infinite cyclic subgroup
$\sca{T_{\del\M}}$ of the mapping class group.

Let $U_{\del\M}$ denote a small collar neighbourhood of $\del\M$ in $\M$;
then $T_{\del\M}$ can be represented by a homeomorphism $f\colon\M\to\M$ fixing $\M\setminus U_{\del\M}$ pointwise.
The inclusion $\M\setminus U_{\del\M}\subset\mathring{\M}$ induces a homotopy equivalence
$F_n(\M\setminus U_{\del\M})\simeq F_n(\mathring{\M})$, so we can study the action
of $f$ on $H_*(F_n(\mathring{\M}))$ by considering the restricted action of $f$ on
$H_*(F_n(\M\setminus U_{\del\M}))$. The latter action is trivial, since $f$ acts as the identity
on $F_n(\M\setminus U_{\del\M})$.
It follows that $T_{\del\M}$ acts trivially on all homology
groups $H_*(F_n(\mathring{\M}))$.

In fact $\sca{T_{\del\M}}$ is contained in $\cJ(2)\subset\Gamma$, since $T_{\del\M}$ is a Dehn twist about a separating curve. In the following we argue that $\sca{T_{\del\M}}$ intersects
$\cJ(3)$ trivially, for all $g\ge1$.

Let $\pi=\pi_1(\M,p_0)$ be generated as a free group by elements
$a_1,b_1,\dots,a_g,b_g$, such that the boundary loop $\del\M$, based at $p_0$, represents
the product of commutators $c:=[a_1,b_1]\dots[a_g,b_g]\in\pi$.

Then $T_{\del\M}$ induces the inner automorphism of $\pi$ given by conjugation by $c$, and similarly,
for $k\in\Z$, the power $T_{\del\M}^k$ induces conjugation by $c^k$ on $\pi$.

In order to show that for $k\neq0$ the mapping class $T_{\del\M}^k$ is not in $\cJ(3)$, it suffices
to check that $T_{\del\M}^k(a_1)\cdot a_1^{-1}$ does not lie in $\gamma_3\pi$.
Since $T_{\del\M}^k(a_1)\cdot a_1^{-1}$ does lie in $\gamma_2\pi$, it suffices to evaluate
$[T_{\del\M}^k(a_1)\cdot a_1^{-1}]$ in the quotient $\gamma_2\pi/\gamma_3\pi$.

Recall that, for $\ell\ge 0$, the quotient $\gamma_\ell\pi/\gamma_{\ell+1}\pi$ is isomorphic to the degree-$(\ell+1)$ part of the free Lie algebra
generated over $\Z$ by elements $A_1,B_1,\dots,A_g,B_g$ sitting in degree 1.
In $\gamma_2\pi/\gamma_3\pi$ we then have the following chain of equalities,
where $C=[A_1,B_1]+\dots+[A_g,B_g]$
denotes the class in $\gamma_1\pi/\gamma_2\pi$ represented by $c$:
\[
\begin{split}
[T_{\del\M}^k(a_1)\cdot a_1^{-1}]&=[c^ka_1c^{-k}a_1^{-1}]\\
 &=[kC,A_1]\\
 &=k\Big([[A_1,B_1],A_1]+[[A_2,B_2],A_1]+\dots+[[A_g,B_g],A_1]\Big).
\end{split}
\]
The last expression is $k$ times a sum of $g$ standard generators of the free abelian group
$\gamma_2\pi/\gamma_3\pi$, hence it does not vanish. This shows that $\sca{T_{\del\M}}$ intersects
$\cJ(3)$ in the trivial group. We conclude this section with a conjecture.
\begin{conj}\label{6.3}
 Let $\M$ be a surface of genus $g\ge1$ with one boundary component, and let $0\le i\le n$.
 Then the kernel of the action of $\Gamma$ on $H_i(F_n(\mathring{\M}))$ is the subgroup
 of $\Gamma$ generated by $\cJ(i)$ and $\sca{T_{\del\M}}$.
\end{conj}

In the case of surfaces without boundary, conceivably the kernel of the mapping class group action on $H_i(F_n(\mathring{\M}))$ is exactly $\cJ(i)$.

\appendix
\section{Reinterpretations of Moriyama's group}
\label{Appendix}
The goal of this appendix is to give alternative descriptions of Moriyama's group $\Mor_n$.
We first show that it is related to arc complexes playing an important role in homological stability and representation stability. Then we compare $\Mor_n$ to certain symmetric sequence valued hypertor groups which measure how to construct configuration spaces of surfaces in terms of configuration spaces of discs.

In this appendix we switch from homeomorphisms to diffeomorphisms of surfaces: this allows us to work
with smoothly embedded arcs rather than topologically embedded arcs, and thus to be in line with most of the literature on the subject. We recall that for a smooth surface $\M$ of genus $g$ with one boundary component,
the mapping class group $\Gamma=\Gamma(\M,\del\M)$ can be equivalently defined as $\pi_0(\Diff(\M,\del\M))$,
where $\Diff(\M,\del\M)$ is the group of diffeomorphisms of $\M$ that fix $\del\M$ pointwise, endowed with the
Whitney $C^{\infty}$-topology.

\subsection{Moriyama's group seen under Poincar\'e--Lefschetz duality}
The following lemma was also observed by Moriyama, and the proof is included for completeness. For an abelian group $A$, we let $A^*$ denote $\Hom(A,\Z)$. See Definition \ref{defMprime} for the definition of $\M'$ and $\I$.

\begin{lem}
\label{lem:naturalisoMoriyama}
There is a natural isomorphism
\[
\Mor_n^* \cong H_n(F_n(\M'),F_{n,1}(\M',\I)).
\]
\end{lem}

\begin{proof}
Since $H_i (\M^n,\Delta_n(\M)\cup A_n(\M))$ vanishes for $i \neq n$, the universal coefficient theorem provides an isomorphism
\[
\Mor_n^*=H_n (\M^n,\Delta_n(\M)\cup A_n(\M))^* \cong H^n(\M^n,\Delta_n(\M)\cup A_n(\M)).
\]
Note that
\[
H^n(\M^n,\Delta_n(\M)\cup A_n(\M)) \cong H^n_c(\M^n \setminus ( \Delta_n(\M)\cup A_n(\M)))=H^n_c(F_n(\M')).
\]
The space $ F_n(\M')$ is a $2n$-manifold with boundary $F_{n,1}(\M',\I)$; by Poincar\'e--Lefschetz duality we obtain an isomorphism
\[
H^n_c(F_n(\M')) \cong H_n(F_n(\M'),F_{n,1}(\M',\I)) .
\]
This establishes the claim. \end{proof}

\subsection{The semi-simplicial space \texorpdfstring{$\bigsqcup_{|S|=\bullet+1} F_{n,S}(\M')$}{U(|S|=bullet+1) F(n,S)(M')}}
Given a set $S \subseteq \{1,\ldots, n\}$, let $F_{n,S}(\M',\I)$ denote the subspace of $F_{n,S}(\M')$ containing
all configurations in which all particles with labels in $S$ lie in $\I$. For $T \subset S$, there are inclusion maps
\[
F_{n,S}(\M') \to F_{n,T}(\M').
\]
These inclusion maps give the assignment
\[
p \longmapsto \bigsqcup_{|S|=p+1} F_{n,S}(\M')
\]
the structure of an augmented semi-simplicial space which we denote by 
\[
\bigsqcup_{|S|=\bullet+1} F_{n,S}(\M') .
\]
Note that the space of $(-1)$-simplices is exactly $F_n(\M')$. The augmentation map from $0$-simplices to $(-1)$-simplices factors as
\[
\bigsqcup_{|S|=0+1} F_{n,S}(\M') \to F_{n,1}(\M',\I) \to F_n(\M').
\]

\begin{lem}
The  map $\Big|\!\Big|\bigsqcup_{|S|=\bullet+1} F_{n,S}(\M') \Big|\!\Big| \to F_{n,1}(\M')$ is a weak equivalence. 
\end{lem}

\begin{proof}
Fix an embedding $[0,1) \times \I$ into $\M'$ sending $\{0\} \times \I$ to $\I$. Let $U_S \subset F_n(\mathring{\M})$ be the subspace of configurations satisfying the following properties:
\begin{itemize}
 \item each point of the configuration with label $i\in S$ is of the form $(t_i,\iota_i) \in [0,1) \times \I $;
 \item for all $i\in S$, no point of the configuration lies on the half-open segment $[0,t_i) \times \{\iota_i\}\subset [0,1)\times\I$.
\end{itemize}
Note that the subspaces $U_{\{i\}}$ form an \emph{open} cover of $F_{n,1}(\M',\I)$, and the nerve of this cover is weakly equivalent to $F_{n,1}(\M',\I)$. Note also that $U_S=\bigcap_{i\in S}U_{\{i\}}$,
hence the nerve of the open cover by the subspaces $U_{\{i\}}$ is the semi-simplicial space
$\bigsqcup_{|S|=\bullet+1} U_S$.

The semi-simplicial space $\bigsqcup_{|S|=\bullet+1} F_{n,S}(\M')$ is the nerve of the cover of the space $F_{n,1}(\M',\I)$ by the subspaces $F_{n,\{i\}}(\M',\I)$; again we have $F_{n,S}(\M')=\bigcap_{i\in S}F_{n,\{i\}}(\M',\I)$.

For all $S$, we have that $F_{n,S}(\M',\I)$ is a subspace of $U_S$ and the inclusion is a homotopy equivalence: a deformation retraction of $U_S$ onto $F_{n,S}(\M',\I)$ is given by pushing each point
$(t_i,\iota_i)$ with label $i\in S$ along the segment $[0,t_i] \times \{\iota_i\}$.
It follows that the inclusion of semi-simplicial spaces
\[
\bigsqcup_{|S|=\bullet+1} F_{n,S}(\M') \subset \bigsqcup_{|S|=\bullet+1} U_S
\]
is a level-wise weak equivalence, yielding a weak equivalence on (thick) geometric realizations. \end{proof}

\subsection{The arc resolution of configuration spaces}
We now recall the arc resolution of \cite{KMcell,MW1}. Given two smooth manifolds $A$ and $B$, let $\Emb(A,B)$ denote the space of smooth embeddings, topologised with the Whitney $C^{\infty}$-topology.

\begin{defn}Let
\[ 
\Arc_p(F_n(\M)) \subset F_n(\mathring \M) \times \Emb(\sqcup_{p+1} [0,1], \M)
\]
be the subspace of sequences of points and smooth arcs $(x_1,\ldots, x_n;\alpha_0,\ldots, \alpha_p)$ satisfying the following conditions:
\begin{itemize} 
  \item $\alpha_i(0) \in \gamma([0,1])$ and $\alpha_i$ is transverse to $\del\M$ at $0$;
  \item $\alpha_i(1) \in \{x_1,\ldots, x_n\}$;
  \item $\alpha_i(t) \notin \del \M \cup \{x_1,\ldots, x_n\}$ for $t \in (0,1)$;
  \item $\alpha_{i}(0)>\alpha_{j}(0)$ whenever $i >j$, using a fixed identification $\I \cong (0,1)$.
\end{itemize}
An element of $\Arc_p(F_n(\M))$ is shown in Figure \ref{ArcResolution}. 
\begin{figure}[!h]    \centering
\includegraphics[scale=.35]{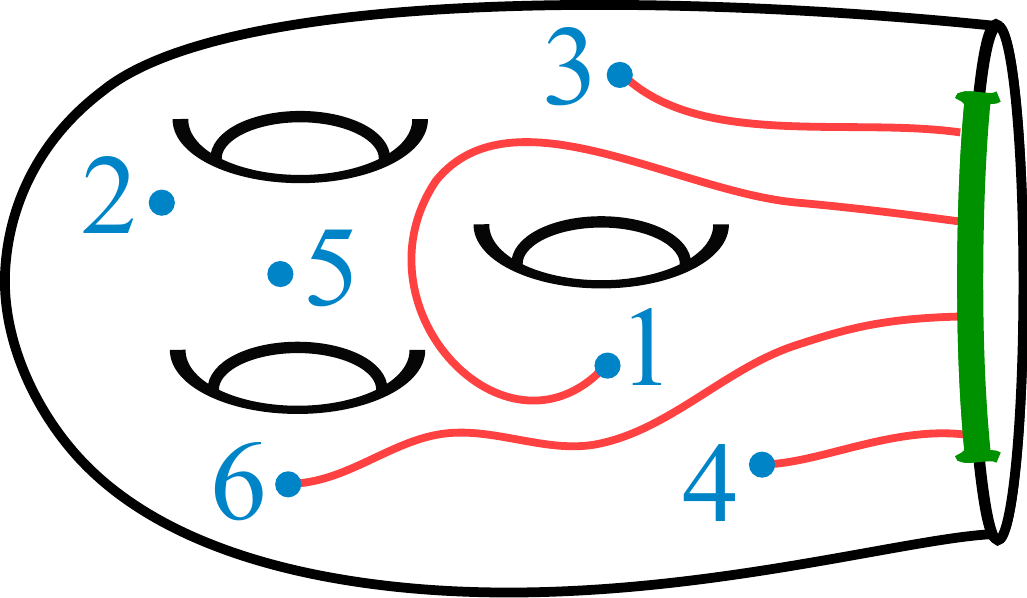} \\ \quad \\ 
\caption{A point in $\Arc_3(F_6(\M))$}
\label{ArcResolution}
\end{figure}  

Let $\Arc^h_p(F_n(\M))$ denote the quotient of $\Arc_p(F_n(\M))$
by the following equivalence relation: $(x_1,\ldots x_n;\alpha_0,\ldots \alpha_p)$ is equivalent
to $(x_1,\ldots, x_n;\alpha'_0,\ldots \alpha'_p)$ if there is a smooth ambient isotopy of
$\M$ fixing at all times
$\set{x_1,\ldots,x_n}$, preserving at all times $\I$, and bringing
$\alpha_i$ to $\alpha'_i$ for all $0\le i\le p$.
We denote by
\[
(x_1,\ldots,x_k;[\alpha_0,\ldots,\alpha_p])\in\Arc_p^h(F_n(\M))
\]
the equivalence class of $(x_1,\ldots x_n;\alpha_0,\ldots \alpha_p)$.
\end{defn}

As $p$ varies, the spaces $\Arc_p(F_n(\M))$ assemble into an augmented semi-simplicial space.
The $i$\textsuperscript{th} face map $d_i: \Arc_p(F_n(\M)) \to \Arc_{p-1}(F_n(\M))$ is given by forgetting the $i$\textsuperscript{th} arc $\alpha_i$. The space $\Arc_{-1}(F_n(\M))$ is homeomorphic to $F_n(\M)$, and so the augmentation map induces a map $|\!|\Arc_\bullet(F_n (\M))|\!| \to F_n(\M)$.

Similarly, the spaces $\Arc^h_p(F_n(\M))$ assemble into an augmented semi-simplicial space;
moreover we have $\Arc^h_{-1}(F_n(\M))\cong \Arc_{-1}(F_n(\M))\cong F_n(\M)$, and we have a natural
map of augmented semi-simplicial spaces
$\Arc_\bullet(F_n(\M))\to \Arc_\bullet^h(F_n(\M))$, given levelwise by the quotient map.

\begin{lem}
There is a zig-zag of level-wise weak equivalences of augmented semi-simplicial spaces between
\[
\bigsqcup_{|S|=\bullet+1} F_{n,S}(\M') \text{ and } \Arc_\bullet(F_n(\M)).
\]
\end{lem}

\begin{proof}
We introduce a ``Moore version'' $\overline{\Arc}_\bullet(F_n(\M))$ of the semi-simplicial space 
$\Arc_\bullet(F_n(\M))$. For $p\ge0$
we define $\overline{\Arc}_p(F_n(\M))$ as the space of sequences
\[
(x_1,\ldots,x_n;(\alpha_0,t_0),\dots,(\alpha_p,t_p)),
\]
where:
\begin{itemize}
 \item $(x_1,\ldots,x_n)\in F_n(\M')$;
 \item $t_0,\dots,t_p$ are elements of the interval $[0,1]$;
 \item $\alpha_i$ is an embedding of the (possibly degenerate) segment $[0,t_i]$ in $\M'$,
 with $\alpha_i(0)\in\I$, $\alpha_i$ transverse to $\del\M$ at 0 whenever $t_i>0$,
 $\alpha_i(s)\in\mathring{\M}$ for $0<s\le t_i$, and $\alpha_i(t_i)\in\set{x_1,\dots,x_n}$;
 we further require that the images of the embeddings $\alpha_0,\dots,\alpha_p$ are disjoint;
  \item $\alpha_{i}(0)>\alpha_{j}(0)$ whenever $i >j$, using a fixed identification $\I \cong (0,1)$.
\end{itemize}
We obtain a zig-zag of levelwise weak equivalences of semi-simplicial spaces
\[
\Arc_\bullet(F_n(\M)) \to \overline{\Arc}_\bullet(F_n(\M)) \leftarrow  \bigsqcup_{|S|=\bullet+1} F_{n,S}(\M')),
\]
by regarding $\Arc_p(F_n(\M))$ as the subspace of $\overline{\Arc}_p(F_n(\M))$ containing sequences for which
all $t_i$'s are equal to $1$, and by regarding $\bigsqcup_{|S|=p+1} F_{n,S}(\M'))$ as the subspace
containing sequences for which all $t_i$'s are equal to $0$.

The space $\overline{\Arc}_p(F_n(\M))$ deformation retracts onto $\bigsqcup_{|S|=p+1} F_{n,S}(\M'))$ by pushing the point
$$(x_1,\ldots,x_n;(\alpha_0,t_0),\dots,(\alpha_p,t_p))$$ 
along the path
$$(\alpha_0((1-s)t_0),\ldots,\alpha_p((1-s)t_p);(\alpha_0|_{[0, (1-s)t_0]},(1-s)t_0),\dots,(\alpha_p|_{[0,(1-s)t_p]},(1-s)t_p)).$$ 
As the parameter $s$ varies from $0$ to $1$ all arcs degenerate to constant paths in $\I$.

Similarly, $\overline{\Arc}_p(F_n(\M))$ can be deformation
retracted onto $\Arc_p(F_n(\M))$ by pushing all values $t_i$ to $1$. 
 Precisely, we first choose a path $\phi_t$ of embeddings  $\M \to \M$ starting at the identity and pushing $\I$ in the interior of $\M$. We then use $\phi_t$ to convert each arc $\alpha_i$ defined on an interval $[0,t_i]$ with $t_i\le 1$, into a new arc $\alpha'_i$ defined on $[0,1]$: we push the image of $\alpha_i$ in the interior of $\M$ using $\phi_t$, and concatenate the  path  $t \mapsto \phi_t(\alpha_i(0))$ for $t \in [0, 1-t_i]$ to obtain  an arc beginning on $\I$ with parameter $1$.
By standard smoothing techniques one can improve the
previous construction so that it yields smooth arcs (also near the concatenation point).
\end{proof}

\subsection{The arc complex of Hatcher--Wahl}
We now recall a semi-simplicial version of the arc complex of Hatcher--Wahl \cite[Section 7]{hatcherwahl} (see also Randal-Williams--Wahl \cite[Section 5.6.2]{RWW}).
Fix distinct points $P_1,\ldots, P_n \in \mathring{\M}$.

\begin{defn}
We define a semi-simplicial set $\mathcal A_n(\M)_\bullet$. The set of $p$-simplices
contains isotopy classes $[\alpha_0,\dots,\alpha_p]$
of collections of disjoint smooth arcs $\alpha_i:[0,1] \to \M'$, for $i=0,\ldots,p$, such that:
\begin{itemize}
\item $\alpha_i(0)\in\I$ and $\alpha_i$ is transverse to $\del\M$ at 0;
\item $\alpha_i(t)\in\mathring{\M}$ for $t>0$;
\item $\alpha_i(1) \in \{P_1,\ldots, P_n \} $;
\item $\alpha_{i}(0)>\alpha_{j}(0)$ whenever $i >j$, using a fixed identification $\I\cong(0,1)$.
\end{itemize}
The face map $d_i$ is given by forgetting the $i$\textsuperscript{th} arc. \end{defn}
An element is shown in Figure \ref{ArcResolutionLine}.
\begin{figure}[!h]    \centering
\includegraphics[scale=.35]{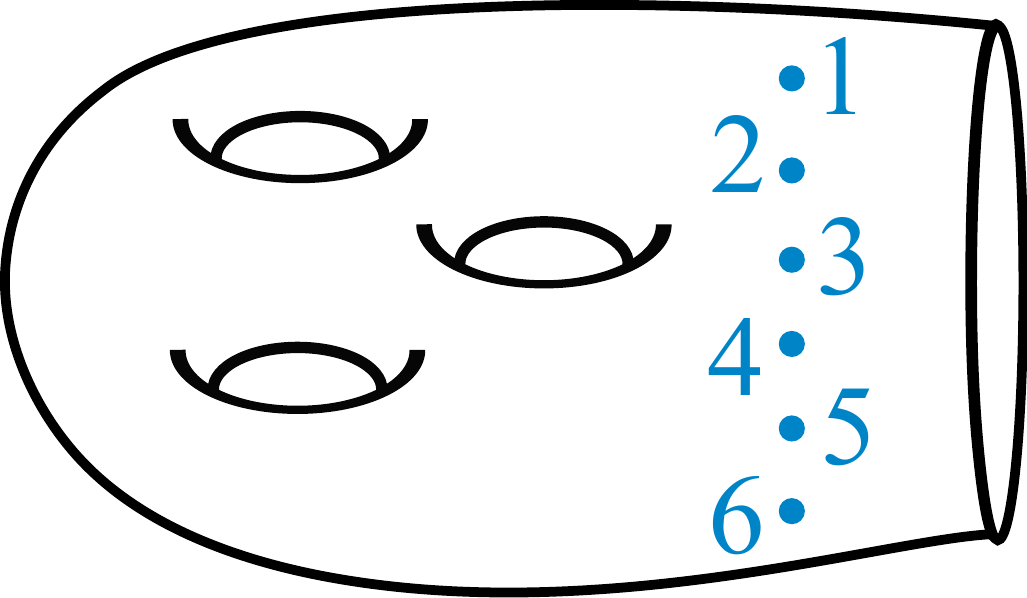} \qquad 
\includegraphics[scale=.35]{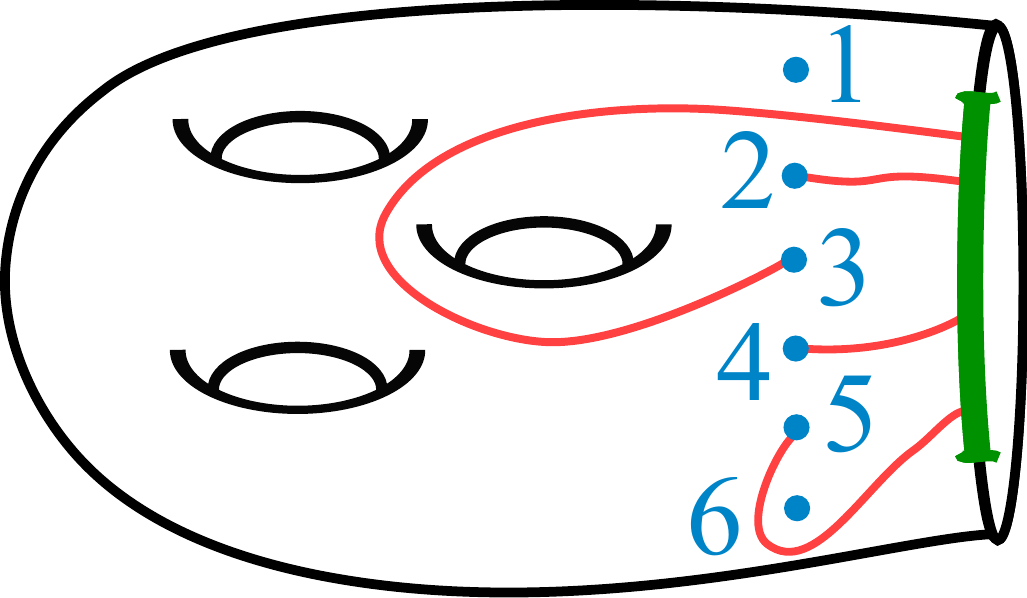} 
\caption{The distinguished points $\{P_1, \ldots, P_n\}$ and a point in $\mathcal A_6(\M)_3$.}
\label{ArcResolutionLine}
\end{figure}  

The semi-simplicial set $\cA_n(\M)_\bullet$ and the semi-simplicial space
$\Arc_\bullet^h(F_n(\M))$ are related as follows. Forgetting all arcs yields, for all $p\ge0$, a map
\[
\fp_p\colon \Arc_\bullet^h(F_n(\M))\to F_n(\mathring{\M}).
\]
The map $\fp_p$ is a covering map,
and $\cA_n(\M)_\bullet$ is defined as the fibre $\fp_p^{-1}(P_1,\dots,P_n)$.

\begin{defn}
 We let $\Diff_0(\M,\del\M)$ denote the connected component of $\Id_{\M}$ in $\Diff(\M,\del\M)$.
 We let $\Diff(\M,\del\M\cup\set{P_1,\dots,P_n})$ be the subgroup of $\Diff(\M,\del\M)$ containing
 diffeomorphisms $f\colon \M\to \M$ fixing $\del\M\cup\set{P_1,\dots,P_n}$ pointwise.

 We define
 \[
 \hPBr_n\colon =\Diff_0(\M,\del\M)\cap \Diff(\M,\del\M\cup\set{P_1,\dots,P_n})\subset\Diff(\M,\del\M).
 \]
\end{defn}
The reason for the notation $\hPBr_n$ is that the pure braid group $\PBr_n$ can be defined
as $\pi_0(\hPBr_n)$. A classical result by Earle and Schatz for mapping class groups of surfaces
with non-empty boundary \cite{EarleSchatz} ensures that
all groups $\Diff(\M,\del\M)$, $\Diff_0(\M,\del\M)$, $\Diff(\M,\del\M\cup\set{P_1,\dots,P_n})$
and $\hPBr_n$ have contractible connected components; moreover the pull-back diagram
\[
\begin{tikzcd}
 \hPBr_n\ar[d,hook]\ar[r,hook] &\Diff_0(\M,\del\M)\ar[d,hook]\\
 \Diff(\M,\del\M\cup\set{P_1,\dots,P_n})\ar[r,hook] & \Diff(\M,\del\M)
\end{tikzcd}
\]
gives, after applying $\pi_0$, a pull-back diagram of discrete groups
\[
\begin{tikzcd}
 \PBr_n\ar[d]\ar[r] & 1\ar[d]\\
 \Gamma(\M,\del\M\cup\set{P_1,\dots,P_n})\ar[r,twoheadrightarrow] & \Gamma,
\end{tikzcd}
\]
which is equivalent to the Birman short exact sequence.
 
Let $/\!/$ denote the homotopy quotient by an action of a (topological) group. We can view $\mathcal A_n(\M)_\bullet$ as an augmented semi-simplicial set with set of $(-1)$-simplices a single point. This lets us view $\mathcal A_n(\M)_\bullet/\!/\PBr_n$ as an augmented semi-simplicial space. Here we consider the natural
action of $\PBr_n=\pi_0(\hPBr_n)$ on each set of homotopy classes of collections of arcs $\cA_n(\M)_p$,
and take levelwise the homotopy quotient.
\begin{lem}
There is a zig-zag of weak equivalences of augmented semi-simplicial spaces between
$\mathcal A_n(\M)_\bullet/\!/\PBr_n$ and $\Arc_\bullet(F_n(\mathring{\M}))$.
\end{lem}

\begin{proof}
We can consider on $\cA_n(\M)_\bullet$ the levelwise action of the topological group $\hPBr_n$, which acts through its quotient $\PBr_n$. We consider the semi-simplicial space $\cA_n(\M)_\bullet/\!/\hPBr_n$, which is levelwise weakly equivalent to $\cA_n(\M)_\bullet/\!/\PBr_n$ because the projection of groups $\hPBr_n\to\PBr_n$ is a weak equivalence.

Next, we observe that for all $p\ge0$, the space
$\Arc^h_p(F_n(\mathring{\M}))$ is homeomorphic to the space
\[
(\Diff_0(\M,\del\M)\times \cA_n(\M)_p)/ \hPBr_n.
\]
Here $\psi\in\hPBr_n$ acts on $\Diff_0(\M,\del\M)$ by $f\mapsto f\circ\psi^{-1}$, and it acts
on $\cA_n(\M)_p$ by $[\alpha_0,\dots,\alpha_p]\mapsto [\psi\circ\alpha_0,\dots,\psi\circ\alpha_p]$.

There is a $\hPBr_n$-invariant map
\[
\rho_p\colon (\Diff_0(\M,\del\M)\times \cA_n(\M)_p)\to \Arc^h_p(F_n(\mathring(\M)))
\]
given by
\[
 \rho_p\colon \pa{f;[\alpha_0,\dots,\alpha_p]}\mapsto \pa{f(P_1),\dots,f(P_n);[f\circ\alpha_0,\dots,f\circ\alpha_p]}.
\]
The map $\rho$ induces a homeomorphism, also denoted
\[
\rho_p\colon (\Diff_0(\M,\del\M)\times \cA_n(\M)_p)/ \hPBr_n\cong \Arc^h_p(F_n(\mathring{\M})).
\]
The homeomorphisms $\rho_p$ give a levelwise homeomorphism of semi-simplicial spaces
\[
\rho\colon (\Diff_0(\M,\del\M)\times \cA_n(\M)_\bullet)/ \hPBr_n\cong \Arc^h_\bullet(F_n(\mathring{\M})).
\]
Since the action of $\hPBr_n$ on $\Diff_0(\M,\del\M)$ is free and admits local sections, and since the topological group
$\Diff_0(\M,\del\M)$ is contractible, the semi-simplicial space
\[
(\Diff_0(\M,\del\M)\times \cA_n(\M)_\bullet)/ \hPBr_n
\]
is a model for the homotopy
quotient $\cA_n(\M)_\bullet/\!/\hPBr_n$.

To complete the proof, we note that the quotient map
$\Arc_p(F_n(\M))\to\Arc_p^h(F_n(\M))$ is also a weak equivalence for all $p\ge0$: this uses
that the quotient map is a fibre bundle, and that the fibre, i.e.
the space of collections of arcs $(\alpha_0,\dots,\alpha_p)$ in a given homotopy
class $[\alpha_0,\dots,\alpha_p]\in\Arc_p^h(F_n(\M))$, is contractible
by Gramain \cite[Theorems 5 and 6]{Gramain}.
\end{proof}

A similar result for unordered configuration spaces is implicit in work of Krannich \cite[Theorem 5.5 and Section 7.3]{Kra}.

\subsection{Moriyama's groups as twisted group homology}
\begin{lem}
There is an isomorphism
\[
H_0(\PBr_n(\M);\tilde H_{n-1}(\mathcal A_n(\M)_\bullet )) \cong H_n\pa{B\PBr_n,|\!| \mathcal A_n(\M)_\bullet/\!/\PBr_n|\!|}.
\]
\end{lem}
\begin{proof}
There is a natural isomorphism 
\[
H_i(B\PBr_n,|\!| \mathcal A_n(\M)_\bullet/\!/\PBr_n|\!|) \cong
H_{i-1}\pa{\PBr_n ; \tCh^{\cell}_*( |\!| \mathcal A_n(\M)_\bullet |\!|)},
\]
where the right hand side denotes the hyperhomology of $\PBr_n$ with twisted coefficients
in the cellular chain complex $\tCh^{\cell}_*( |\!| \mathcal A_n(\M)_\bullet |\!|)$.

By Hatcher--Wahl \cite[Proposition 7.2]{hatcherwahl}, $\widetilde H_i(\mathcal A_n(\M)_\bullet )$ vanishes except for $i=n-1$. Thus the hyper-homology spectral sequence collapses, showing
\[
H_{i+n-1}\pa{\PBr_n ; \tCh^{\cell}_*( |\!| \mathcal A_n(\M)_\bullet |\!|)} \cong H_{i}(\PBr_n;\widetilde H_{n-1}(\mathcal A_n(\M)_\bullet )).
\]
Combining these two isomorphisms gives the result.\end{proof}

Stringing together all the previous lemmas gives the following theorem.

\begin{thm}
\label{thm:Moriyamaarccomplex}
There is an isomorphism
\[
\Mor_n^* \cong H_0(\PBr_n(\M);\widetilde H_{n-1}(\mathcal A_n(\M)_\bullet )).
\]
\end{thm}

\subsection{Moriyama's groups as hypertor groups}
In this subsection, we note that the groups $\Mor_n^*$ can be viewed as certain symmetric sequence valued hypertor groups. Recall that the category of symmetric sequences of chain complexes is endowed with a symmetric monoidal product called Day convolution. This product lets us make sense of monoid objects, modules over monoid objects, tensor products over monoid objects, Tor groups, algebras over operads, etc.

The sequence of chain complexes $\{\Ch_*(F_n(\R^2))\}_n$ has the structure of an $E_2$-algebra in symmetric sequences of chain complexes;
in particular, it is an $E_1$-algebra and it
is equivalent after rectification to a monoid object, which we denote by $\Ch_*(F(\R^2))$.
Picking an embedding $\R^n \cup \mathring{\M}  \hookrightarrow \mathring{\M}$
gives maps 
\[
F_n(\R^2) \times F_m(\mathring{\M}) \to F_{n+m}(\mathring{\M}).
\]
After rectification, these maps give rise to a $\Ch_*(F(\R^2))$-module $\Ch_*(F(\mathring{\M}))$, whose $n$\textsuperscript{th} component is equivalent to $\Ch_*(F_n(\mathring{\M}))$.

Regard $\Z$ as a symmetric sequence of chain complexes concentrated in homological degree $0$. It is a $\Ch_*(F(\R^2))$-module via the augmentation
$\Ch_*(F(\R^2))\to\Z$, which is non-trivial only on the $0$\textsuperscript{th} component, where it induces an isomorphism
$H_0(\Ch_*(F_0(\R^2))) \overset{\cong}{\to} \Z$. View $\Mor_n^*$ as a symmetric sequence of abelian groups by viewing all but the $n$\textsuperscript{th} component as $0$. After applying Poincar\'e-Lefschetz daulity, Proposition \ref{prop:main} can be rephrased as saying that $ \Ch_*(F(\mathring{\M}))$ has a resolution by free $\Ch_*(F(\R^2))$-modules with $n$\textsuperscript{th} syzygy given by $\Mor_n^*$ in homological degree $n$. This gives the following corollary.

\begin{cor}
 \label{cor:hypertor}
There is an isomorphism of symmetric sequences of abelian groups 
 \[
 \Tor^{\Ch_*(F(\R^2))}_n( \Ch_*(F(\mathring{\M})) ,\Z) \cong \Mor_n^*.
 \]
\end{cor}

Hypertor groups of this form are often referred to as ``derived indecomposables''.

\subsection{Some questions}
We close the article with two questions.
\begin{question}
Starting with the Solomon--Tits theorem \cite{Solomon}, there is a long literature on generating sets for the top homology of spherical simplicial complexes (see e.g. Broaddus \cite[Theorem 1.2]{Broaddus} and \cite[Theorem 2.40]{MW1}). Can one describe a generating set for $\widetilde H_{n-1}(\mathcal A_n(\M)_\bullet )$? Does this shed light on the homology of configuration spaces?
\end{question}

\begin{question}
The complexes $\mathcal A_n(\M)_\bullet $ and $\Arc_\bullet(F_n(\mathring{\M}))$ have primarily been studied from the perspective of homological stability and representation stability \cite{hatcherwahl,RWW,MW1}. Is there a way of deducing Theorem \ref{thm:main} using stability?

For instance, representation stability for the sequence $\{H_1(F_n(\mathring{\M}))\}_n$, together with the fact that $\M$ has genus at least $1$,
implies that there is a surjection
\[
\Ind_{\fS_1 \times \fS_{n-1}}^{\fS_n}H_1(F_1(\mathring{\M})) \to H_1(F_n(\mathring{\M})).
\]
Since $\cJ(1)$ acts trivially on $H_1(F_1(\mathring{\M})) \cong H_1(\M)$, one can conclude
that $\cJ(1)$ acts trivially on $H_1(F_n(\mathring{\M}))$ as well.

It is also easy to see that $\cJ(i)$ acts trivially on $H_i(F_i(\M))$
for $i\ge0$, using that $H_i(F_i(\M))$ injects into $\Mor_i^*$. Can one show $\cJ(i)$ acts trivially on $H_i(F_n(\mathring{\M}))$ for all $n$ by constructing $H_i(F_n(\mathring{\M}))$ out of classes
in $H_a(F_b(\R^2))$ and in $H_c(F_c(\M))$ with a sensible set of operations?
This should be related to the fact that both Moriyama's space $\M^n/(\Delta_n(\M)\cup A_n(\M))$
and the arc complex $\cA_n(\M)_\bullet$ are spherical.
\end{question}

\bibliographystyle{amsalpha}
\bibliography{bibliography}
\newpage
\end{document}